\documentclass[a4paper,11pt,reqno]{amsart}

\usepackage{geometry}
\usepackage[utf8]{inputenc}
\usepackage[T1]{fontenc}
\usepackage[english]{babel} 

\usepackage{amsmath,amsfonts,amssymb,amsthm}
\usepackage{mathtools}
\usepackage{booktabs,array}
\usepackage{enumitem}
\usepackage[linesnumbered,lined,boxed,commentsnumbered]{algorithm2e}
\usepackage{graphicx}
\usepackage{xcolor}
\usepackage{hyperref}
\usepackage{natbib} 
\usepackage{tikz}
\usetikzlibrary{positioning,arrows.meta}

\newtheorem{definition}{Definition}[section]
\newtheorem{theorem}[definition]{Theorem}
\newtheorem{example}[definition]{Example}
\newtheorem{proposition}[definition]{Proposition}
\newtheorem{lemma}[definition]{Lemma}

\newtheorem{observation}[definition]{Note}

\definecolor{myblue}{RGB}{0, 60, 160}      

\let\origmaketitle\maketitle
\def\maketitle{
	\begingroup
	\def\uppercasenonmath##1{} 
	\let\MakeUppercase\relax 
	\origmaketitle
	\endgroup
}

\begin{document}

\title[Optimization of Growth Bounds over Autocatalytic Subnetworks]{\Large Optimization of Kinetic--Stoichiometric Growth Bounds over Autocatalytic Subnetworks}

\author[V. Blanco and G. Gonz\'alez]{
{\large V\'ictor Blanco$^{\dagger}$ and Gabriel Gonz\'alez$^{\ddagger}$}\bigskip\\
\small $^\dagger$Institute of Mathematics (IMAG), Universidad de Granada, Spain\\
\small $^\ddagger$Instituto Tecnológico de Aeronáutica (ITA), Brazil\bigskip\\
\small \texttt{vblanco@ugr.es}, \texttt{gabgondom@ita.br}
}

\maketitle

\begin{abstract}
Autocatalytic subnetworks play a central role in chemical
reaction systems, where they govern the emergence of sustained production and
balanced growth. Such systems can be naturally represented as directed
multihypergraphs, providing a unified framework to describe both
stoichiometric structure and reaction kinetics. Existing optimization
approaches identify autocatalytic subnetworks by maximizing their structural
amplification, typically quantified through the Maximum Amplification Factor
(MAF). However, structural amplification alone does not determine the growth
potential of a reaction network, since kinetically slower subnetworks may
exhibit larger amplification factors than faster ones. Motivated by recent
kinetic--topological bounds for balanced growth under scalable dynamics, we
study the problem of selecting the autocatalytic subnetwork that maximizes
the growth bound
$(\alpha^*-1)\|\mathbb{S}|_{\mathcal A'}\|_\kappa$,
where $\alpha^*$ denotes the MAF and
$\|\mathbb{S}|_{\mathcal A'}\|_\kappa$ is a kinetic consumption norm. The
resulting optimization problem is formulated as a mixed-integer bilinear model
combining hypergraph selection with generalized fractional amplification
constraints. Exploiting the discrete structure of the kinetic norm, we develop
an exact parametric solution method in which each subproblem reduces to a MAF
maximization problem solved through a Dinkelbach-type generalized fractional
programming algorithm. We establish finite convergence and global optimality
of the proposed approach. Applications to the Oregonator, the formose
reaction, carbon-fixation cycles, and additional benchmark networks
demonstrate when maximizing the proposed kinetic--stoichiometric growth bound
differs from classical MAF maximization and clarify the interplay between
stoichiometric amplification, reaction kinetics, and realized balanced growth.
The proposed framework provides an exact optimization methodology for
identifying reaction subnetworks with the greatest theoretical growth
potential, thereby establishing a direct link between stoichiometric
amplification, reaction kinetics, and balanced-growth analysis within a
unified multihypergraph formulation.
\end{abstract}
 

\section{Introduction}\label{sec:intro}

Understanding how chemical reaction networks sustain self-amplification and balanced growth is a fundamental problem in mathematical chemistry and process systems engineering. More broadly, similar questions arise in systems biology, economics, and network science, where interacting components collectively reinforce their own production. In chemical reaction networks, one seeks to identify subsets of molecular species and reactions capable of collectively sustaining and accelerating their own production, a question that lies at the heart of autocatalysis, systems chemistry, metabolic engineering, and origin-of-life research~\citep{kauffman1986,hordijk2004detecting,hordijk2018autocatalytic,feinberg2019foundations}. Closely related mathematical questions arise in production economies, supply chains, innovation systems, and ecological networks, where groups of interacting entities jointly reinforce the production of goods, resources, or information~\citep{neumann1945model,aghion2008economics,huang2014ecological,acemoglu2012network,acemoglu2017networks,baqaee2019granular,nagurney2013supply}. Despite their different interpretations, these systems share a common mathematical structure in which multiple entities interact simultaneously to transform collections of inputs into collections of outputs, giving rise to self-reinforcing growth mechanisms. From a mathematical perspective, this naturally leads to the optimization problem of identifying the reaction subnetwork with the greatest intrinsic growth potential.

From a chemical engineering perspective, selecting reaction pathways requires connecting network structure with kinetic performance. This connection is also relevant to bioprocess intensification, where biocatalyst engineering and process engineering offer complementary opportunities for improving chemical production~\citep{WOODLEY2017297}. These considerations motivate computational tools for evaluating the growth potential of alternative autocatalytic subnetworks before more detailed kinetic and process-level assessment.

From an optimization perspective, the identification of functional subnetworks within large interaction networks constitutes a challenging class of combinatorial optimization problems. 
Similar formulations arise in transportation, logistics, machine learning, and systems biology, where one seeks to extract the subset of components and interactions that best satisfies a prescribed performance criterion, such as connectivity, robustness, resilience, or efficiency. Optimization methodologies of this type have been successfully applied to transportation, logistics, communication, energy networks, and many other large-scale networked systems~\citep{magnanti1984network,ahuja1988network,LS2019,blanco2025fixed}. More recently, similar optimization paradigms have also become increasingly important in machine learning, where combinatorial formulations support feature selection, sparse learning, graph learning, interpretable models, and neural architecture design~\citep{bertsimas2016best,atamturk2020strong,fischetti2018deep,labbe2019mixed}. In chemical reaction networks, optimization has traditionally focused on metabolic flux analysis, pathway identification, strain design, and stoichiometric feasibility~\citep{orth2010reconstruction,burgard2003optknock,klamt2009hypergraphs}. Comparatively less attention has been devoted to optimization models that directly identify reaction subnetworks responsible for self-amplification and sustained growth.

Mixed-integer optimization has also been used to infer chemical reaction networks from experimental data. \citet{WILLIS201631} developed a mixed-integer linear programming approach for selecting sparse reaction networks from a superstructure, while \citet{WILSON201988} combined reaction-network identification and kinetic parameter estimation in a mixed-integer nonlinear formulation with adaptive experimental design. The present work addresses a complementary selection problem: given a reaction network and its kinetic constants, we optimize a growth-bound criterion over its autocatalytic subnetworks. Computational studies of autocatalytic kinetics also include the stochastic master-equation approach of \citet{LENTE2019587}, which estimates time-dependent state distributions. Here, the focus is on deterministic balanced-growth bounds and subnetwork optimization.

Only recently have optimization-based methodologies begun to address self-amplification directly. Exact optimization models have been proposed for identifying stoichiometrically autocatalytic subnetworks through the Maximum Amplification Factor (MAF), thereby avoiding the combinatorial enumeration of amplification structures~\citep{gonzalez2024maf,Gagrani2024}. These ideas have subsequently been extended to the construction of autocatalytic hypergraphs under multiperiod growth dynamics~\citep{blanco2026constructing}. Nevertheless, these optimization models remain fundamentally structural because the MAF depends exclusively on stoichiometric information and neglects the kinetic rates governing how rapidly reactions proceed. Consequently, two subnetworks with identical stoichiometric amplification but substantially different kinetics receive the same score, even though they may exhibit very different growth rates.

This limitation has recently been addressed theoretically by
\citet{gagrani2026topological}, who derived a kinetic--stoichiometric
upper bound on the balanced-growth rate of scalable reaction networks.
The bound combines the Maximum Amplification Factor (MAF), which measures
the stoichiometric self-amplification of a reaction subnetwork, with a
kinetic consumption norm determined by the reaction rates. It therefore
shows that attainable growth depends simultaneously on stoichiometric
amplification and reaction kinetics. Consequently, maximizing the MAF
alone does not necessarily maximize the theoretical growth rate.
Although the bound is generally not attained, it provides a computable
surrogate for the growth potential of a reaction subnetwork, naturally
suggesting the optimization problem studied in this paper.

Accordingly, we consider the optimization problem of maximizing the kinetic--stoichiometric growth bound
$(\alpha^*-1)\|\mathbb S|_{\mathcal A'}\|_\kappa$
over all autocatalytic reaction subnetworks.

The purpose of this optimization problem is to identify the autocatalytic subnetwork with the largest theoretical growth potential. Its optimal value provides an upper bound on the balanced-growth rate attainable by any autocatalytic subnetwork of the network. This bound depends only on the stoichiometric data and kinetic constants, although the subnetwork attaining it need not maximize the realized growth rate.

For an open reactor with fixed external concentrations and a uniform first-order dilution rate $D$, this bound has a useful exclusion interpretation. Writing $B^*$ for the maximum growth bound, $B^*\leq D$ rules out strictly positive net balanced growth for every autonomous autocatalytic subnetwork governed by the scalable kinetics considered here and admitting balanced growth with strictly positive species composition. Conversely, $B^*>D$ does not establish persistence: a candidate must also admit a positive balanced-growth state with realized growth rate exceeding $D$.

To solve this problem exactly, we exploit the discrete structure of the kinetic norm. Since it assumes only finitely many values over all feasible subnetworks, the original optimization problem reduces to a finite family of generalized fractional programs, each associated with a fixed kinetic level. Each subproblem is solved exactly through a Dinkelbach-type algorithm~\citep{dinkelbach1967,crouzeix1985}, while incumbent-based bounds discard kinetic levels that cannot improve the current best solution. The resulting algorithm is finite and returns a globally optimal reaction subnetwork.

The proposed methodology is evaluated on five representative chemical reaction networks exhibiting distinct interactions between stoichiometric amplification and reaction kinetics. The computational results illustrate how incorporating kinetic information can alter the optimal autocatalytic subnetwork and identify situations in which the kinetic--stoichiometric growth bound provides either an accurate or a conservative surrogate for the actual growth dynamics.

The main contributions of this paper are threefold. First, we formulate the optimization of the kinetic--stoichiometric growth bound over reaction subnetworks, extending the theoretical framework of \citet{gagrani2026topological} from the analysis of fixed reaction networks to the identification of optimal subnetworks. Second, we develop an exact solution algorithm based on finite parametric enumeration and generalized fractional programming, and establish its finite convergence. Third, we analyze the resulting optimization criterion on representative chemical reaction networks, illustrating both its predictive capabilities and its limitations as a surrogate for balanced growth.

The rest of the paper is organized as follows. Section~\ref{sec:prelim} introduces reaction hypergraphs, autocatalytic subnetworks, the MAF, and the kinetic growth bounds on which the paper builds. Section~\ref{sec:problema} formulates the bound-optimal subnetwork problem.
The solution methodology, based on a finite descending sweep over scaled
kinetic-norm values combined with a Dinkelbach-type generalized fractional
programming scheme, is presented in Section~\ref{sec:sweep}.
Section~\ref{sec:examples} provides two illustrative
examples: a designed instance showing that MAF maximization and growth-bound
maximization may select different subnetworks, and the Oregonator benchmark,
which validates the method on a chemically meaningful autocatalytic system.
Section~\ref{sec:exp} reports the computational experiments on the formose
reaction, the rTCA and glyoxylate cycles, and synthetic scalability instances.
Finally, Section~\ref{sec:conc} concludes and discusses future research
directions.

\section{Preliminaries}\label{sec:prelim}

Chemical reaction networks can be naturally represented as directed multihypergraphs, where molecular species correspond to nodes and reactions correspond to hyperarcs connecting multisets of reactants to multisets of products. This representation captures the inherently higher-order nature of chemical reactions, in which multiple reactants may simultaneously produce multiple products, and provides a convenient mathematical framework for studying stoichiometric structure, autocatalysis, and growth phenomena \citep{feinberg2019foundations,gonzalez2024maf}. More generally, the same formalism encompasses production systems, economic input--output models, and technological networks, where nodes represent goods, resources, or technologies, and hyperarcs encode transformation processes \citep{klamt2009hypergraphs,battiston2021physics,bick2023higher}.

\begin{definition}\label{def:hypergraph}
A \emph{directed multihypergraph} is a pair
$\mathcal{H}=(\mathcal{N},\mathcal{A})$, where $\mathcal{N}$ is the set of
nodes and $\mathcal{A}$ is the set of hyperarcs. Each hyperarc
$a=(S_a,T_a)\in\mathcal{A}$ consists of a source multiset $S_a$ and a target
multiset $T_a$ over $\mathcal{N}$, representing the entities consumed and
produced by the corresponding transformation process, respectively.
Multiplicities within $S_a$ and $T_a$ account for stoichiometric coefficients,
input--output coefficients, or interaction intensities.

For each node $v\in\mathcal{N}$ and hyperarc $a\in\mathcal{A}$, let
$\mathbb{S}_{va}$ and $\mathbb{T}_{va}$ denote the multiplicities of $v$ in
$S_a$ and $T_a$, respectively. The resulting source and target incidence
matrices satisfy
$\mathbb{S},\mathbb{T}\in\mathbb{R}_{+}^{|\mathcal{N}|\times|\mathcal{A}|}$,
and define the \emph{stoichiometric} (or \emph{net incidence}) matrix
\[
\mathbb{Q}=\mathbb{T}-\mathbb{S}.
\]
\end{definition}

Chemical reaction networks constitute a fundamental example of directed
multihypergraphs. In this setting, the nodes of $\mathcal{H}$ correspond to
chemical species, while each hyperarc represents a chemical reaction. The
source and target multisets encode the reactants and products, respectively,
and the multiplicities of the species in these multisets coincide with their
stoichiometric coefficients. Consequently, the incidence matrices
$\mathbb{S}$ and $\mathbb{T}$ are precisely the reactant and product
stoichiometric matrices of the network, and
$\mathbb{Q}=\mathbb{T}-\mathbb{S}$ is the corresponding stoichiometric matrix.
This hypergraph representation is mathematically equivalent to the standard
representation of a chemical reaction network, while naturally extending to
other systems involving higher-order interactions.

\begin{example}
To illustrate this representation, consider a simple chemical reaction network
with three molecular species:
$A$, $B$, and $C$. The network is represented by the directed multihypergraph
$\mathcal{H}=(\mathcal{N},\mathcal{A})$, where
$\mathcal{N}=\{A,B,C\}$ is the set of species and
$\mathcal{A}=\{a_1,a_2,a_3\}$ is the set of reactions. Specifically,
\begin{align*}
a_1 &= (\{A,A\},\{B\})
&& \text{(dimerization: two molecules of $A$ produce one molecule of $B$)},\\
a_2 &= (\{B\},\{C,C\})
&& \text{(one molecule of $B$ produces two molecules of $C$)},\\
a_3 &= (\{C\},\{A\})
&& \text{(one molecule of $C$ is converted into one molecule of $A$)}.
\end{align*}
For this reaction network, the source and target incidence matrices are
given by
$$
\mathbb{S}=
\begin{pmatrix}
2 & 0 & 0\\
0 & 1 & 0\\
0 & 0 & 1
\end{pmatrix},
\qquad
\mathbb{T}=
\begin{pmatrix}
0 & 0 & 1\\
1 & 0 & 0\\
0 & 2 & 0
\end{pmatrix},
$$
where the rows correspond to the molecular species
$(A,B,C)$ and the columns correspond to the reactions
$(a_1,a_2,a_3)$. Thus, for example,
$\mathbb{S}_{A,a_1}=2$ indicates that reaction $a_1$ consumes two molecules of
species $A$, whereas $\mathbb{T}_{C,a_2}=2$ indicates that reaction $a_2$
produces two molecules of species $C$. The associated stoichiometric matrix is
$$
\mathbb{Q}=\mathbb{T}-\mathbb{S}=
\begin{pmatrix}
-2 & 0 & 1\\
1 & -1 & 0\\
0 & 2 & -1
\end{pmatrix}.
$$  
Figure~\ref{fig:bipartite-crn} depicts the corresponding bipartite
representation of the reaction network. Circular nodes represent the
molecular species, rectangular nodes represent the reactions, and directed
edges connect reactant species to reactions and reactions to product
species. Edge labels indicate stoichiometric coefficients whenever they are
greater than one, making the consumption and production of multiple
molecules explicit.
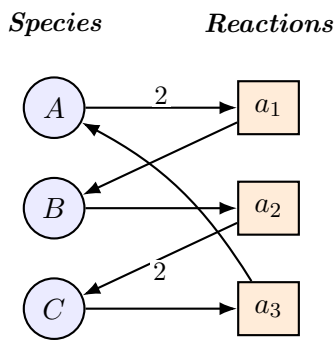
\begin{figure}[ht]
\centering
\begin{tikzpicture}[
    >=Latex,
    species/.style={
        circle,
        draw,
        thick,
        minimum size=8mm,
        inner sep=1pt,
        fill=blue!8
    },
    reaction/.style={
        rectangle,
        draw,
        thick,
        minimum width=8mm,
        minimum height=7mm,
        fill=orange!15
    },
    edge/.style={
        ->,
        thick
    },
    coeff/.style={
        font=\small,
        fill=white,
        inner sep=1pt
    },
    node distance=.5cm and 2cm
]

\node[species] (x1) {$A$};
\node[species, below=of x1] (x2) {$B$};
\node[species, below=of x2] (x3) {$C$};

\node[reaction, right=of x1] (r1) {$a_1$};
\node[reaction, right=of x2] (r2) {$a_2$};
\node[reaction, right=of x3] (r3) {$a_3$};

\draw[edge] (x1) -- node[coeff, above] {$2$} (r1);
\draw[edge] (r1) -- (x2);

\draw[edge] (x2) -- (r2);
\draw[edge] (r2) -- node[coeff, below] {$2$} (x3);

\draw[edge] (x3) -- (r3);
\draw[edge] (r3) to[bend right=15] (x1);

\node[font=\small\bfseries, above=4mm of x1] {Species};
\node[font=\small\bfseries, above=5mm of r1] {Reactions};

\end{tikzpicture}
\caption{Bipartite representation of the chemical reaction network $2A\to B$, $B\to 2C$, and $C\to A$. Edge labels indicate
stoichiometric coefficients greater than one.}
\label{fig:bipartite-crn}
\end{figure}
\end{example}
\begin{definition}
A \emph{hyperflow} is a vector
$\mathbf{f}\in\mathbb{R}_{+}^{|\mathcal{A}|}$ assigning a nonnegative
intensity to each hyperarc. The corresponding net production vector is
$$
\mathbf{x}=\mathbb{Q}\mathbf{f}\in\mathbb{R}^{|\mathcal{N}|},
$$
whose positive and negative components represent the net production and net
consumption of the corresponding nodes, respectively. When
$\mathcal{H}$ represents a chemical reaction network, a hyperflow coincides
with the vector of reaction rates, and $\mathbf{x}$ gives the net production
rates of the molecular species.
\end{definition}

For the reaction network above, consider the hyperflow
$\mathbf{f}=(1,3,2)^\top$, meaning that reactions $a_1$, $a_2$, and $a_3$
occur with reaction rates $1$, $3$, and $2$, respectively. The resulting net
production vector is
$$
\mathbf{x}=
\begin{pmatrix}
x_A\\
x_B\\
x_C
\end{pmatrix}
=
\mathbb{Q}\mathbf{f}
=
\begin{pmatrix}
0\\
-2\\
4
\end{pmatrix}.
$$
Thus, species $A$ is in balance, species $B$ undergoes a net consumption of
two molecules, and species $C$ has a net production of four molecules.

More generally, let $\mathbf{n}\in\mathbb{R}_{>0}^{|\mathcal{N}|}$ denote the
state vector and let
$\mathbf{J}(\mathbf{n})\in\mathbb{R}_{+}^{|\mathcal{A}|}$ be the hyperflow
induced by the interaction law. The evolution of the system is governed by
\begin{equation}
\dot{\mathbf{n}}=\mathbb{Q}\mathbf{J}(\mathbf{n}),
\label{eq:ode}
\end{equation}
where $\mathbb{Q}$ is the stoichiometric matrix. When $\mathcal{H}$
represents a chemical reaction network, $\mathbf{J}(\mathbf{n})$ is the vector
of reaction rates, and~\eqref{eq:ode} reduces to the classical
stoichiometric equations. More generally, $\mathbf{J}(\mathbf{n})$
describes the hyperflow generated by the transformation processes associated
with the hyperarcs.

We now introduce the main structural object studied in this paper: 
\emph{autocatalytic subnetworks}.

\begin{definition}
\label{def:autoamplificante}
A subhypergraph
$\mathcal{H}'=(\mathcal{N}',\mathcal{A}')\subseteq\mathcal{H}$ is
\emph{autocatalytic} with respect to
$\mathcal{M}\subseteq\mathcal{N}'$ if:
\begin{enumerate}
\item (\emph{Self-sufficiency}) Every hyperarc in $\mathcal{A}'$ has at least
one source node and one target node in $\mathcal{M}$, and every node in
$\mathcal{M}$ appears as both a source and a target node of some hyperarc in
$\mathcal{A}'$.

\item (\emph{Net positive realizability}) There exists a strictly positive
hyperflow
$\mathbf{f}\in\mathbb{R}_{>0}^{|\mathcal{A}'|}$ such that
$$
\mathbb{Q}_{\mathcal{M}}\mathbf{f}>0
$$
componentwise, where
$\mathbb{Q}_{\mathcal{M}}\in\mathbb{R}^{|\mathcal{M}|\times|\mathcal{A}'|}$
denotes the submatrix of $\mathbb{Q}$ with rows indexed by $\mathcal{M}$ and
columns indexed by the hyperarcs of the subhypergraph, $\mathcal{A}'$,  that is,
$(\mathbb{Q}_{\mathcal{M}}\mathbf{f})_v=\sum_{a\in\mathcal{A}'}\mathbb{Q}_{va}f_a$
for every $v\in\mathcal{M}$.
\end{enumerate}
\end{definition}
Self-sufficiency means that the subnetwork is closed with respect to $\mathcal{M}$: every selected reaction consumes and produces species of $\mathcal{M}$, and every species of $\mathcal{M}$ is both consumed and produced within the subnetwork. Net positive realizability means that the reactions can operate at positive rates while simultaneously producing every species in $\mathcal{M}$. Together, these conditions correspond to the stoichiometric notion of autocatalysis of~\citet{blokhuis2020universal}. When $\mathcal{M}=\mathcal{N}'$, self-sufficiency coincides with the notion of \emph{autonomy} used in \citep{blokhuis2020universal,gagrani2026topological}.

Although introduced here in the context of chemical reaction networks,
the notion of an autocatalytic subnetwork, as considered in \citep{blokhuis2020universal,Gagrani2024}, is equally applicable to a
broader class of directed multihypergraphs. In particular, it captures
self-amplifying structures~\citep{gonzalez2024maf,blanco2026constructing} arising in von Neumann production models,
economic production networks, information propagation, and other systems
whose dynamics are governed by higher-order interactions. The
amplification capability of such structures is quantified through the
Maximum Amplification Factor (MAF), introduced by
\citet{gonzalez2024maf}.

\begin{definition}
\label{def:maf}
Let $\mathcal{H}'=(\mathcal{N}',\mathcal{A}')$ be an autocatalytic subhypergraph, let $\mathcal{M}\subseteq\mathcal{N}'$, and let
$\mathbf{f}\in\mathbb{R}^{|\mathcal{A}'|}_{>0}$ be a strictly positive
hyperflow. The \emph{amplification factor} is defined as
$$
\alpha(\mathcal{H}',\mathcal{M};\mathbf{f})
=
\min_{v\in\mathcal{M}}
\frac{\displaystyle\sum_{a\in\mathcal{A}'}\mathbb{T}_{va}f_a}
{\displaystyle\sum_{a\in\mathcal{A}'}\mathbb{S}_{va}f_a},
$$
and the corresponding \emph{Maximum Amplification Factor} (MAF) is
$$
\alpha^*(\mathcal{H}',\mathcal{M})
=
\sup_{\mathbf{f}\in\mathbb{R}_{>0}^{|\mathcal{A}'|}}
\alpha(\mathcal{H}',\mathcal{M};\mathbf{f}).
$$
\end{definition}

\begin{theorem}[\cite{gonzalez2024maf}]
\label{thm:maf_char}
A subhypergraph $\mathcal{H}'$ is autocatalytic with respect to
$\mathcal{M}$ if and only if
$$
1<\alpha^*(\mathcal{H}',\mathcal{M})<\infty.
$$
\end{theorem}

The computation of the MAF reduces to a generalized fractional program that
can be solved exactly through a Dinkelbach-type algorithm
\citep{crouzeix1985}, while identifying an autocatalytic subhypergraph
with maximum MAF leads to a mixed-integer nonlinear optimization problem
(MINLP) \citep{gonzalez2024maf}. For minimal autonomous subhypergraphs,
the MAF coincides with the dominant generalized eigenvalue of
$$
\mathbb{T}\mathbf{f}
=
\alpha^*\mathbb{S}\mathbf{f},
$$
which is analogous to the Collatz--Wielandt characterization of the Perron
spectral radius \citep{gagrani2026topological}. In the context of chemical
reaction networks, this quantity has been characterized for the five
minimal autocatalytic cores identified by
\citet{blokhuis2020universal}, whose corresponding MAF values are reported
in Table~\ref{tab:nucleos}. These cores constitute the smallest
irreducible autocatalytic motifs, and every autocatalytic reaction network
contains at least one of them~\citep{blokhuis2020universal}.

\begin{table}[h]
\centering
\begin{tabular}{lccccc}
\toprule
& Type I & Type II & Type III & Type IV & Type V \\
\midrule
$\alpha^*$ & $\sqrt{2}$ & $\approx 1.32$ & $\sqrt{2}$ &
  $\varphi\approx1.618$ & $2$ \\
\bottomrule
\end{tabular}
\caption{MAF of the five minimal autocatalytic cores. 
  $\varphi=(\sqrt5+1)/2$ is the golden ratio.}
\label{tab:nucleos}
\end{table}
The MAF is a \emph{purely structural} quantity, since it depends only on
the incidence (stoichiometric) matrices $\mathbb{S}$ and $\mathbb{T}$ and is therefore
independent of the interaction law governing the hyperflows. Consequently,
different dynamical systems sharing the same multihypergraph have identical
MAF values, even though their temporal evolution may differ
substantially. Bridging this gap between structural amplification
potential and dynamic realization constitutes the main motivation of this
work. To this end, we build upon the kinetic bounds established by
\citet{gagrani2026topological} for scalable dynamics.

Following \citet{gagrani2026topological}, we consider \emph{scalable dynamics}, namely interaction laws satisfying \[ \mathbf J(c\,\mathbf n)=c\,\mathbf J(\mathbf n), \qquad \forall\, c>0. \] For chemical reaction networks, this property is realized by the linear-volume mass-action law \begin{equation} J_a(\mathbf n) = k_a\,\Omega \prod_{v\in\mathcal N} \left(\frac{n_v}{\Omega}\right)^{\mathbb S_{va}}, \qquad \Omega=\mathbf w^\top\mathbf n, \label{eq:linear-volume} \end{equation} which is homogeneous of degree one for reactions of arbitrary order. For each hyperarc $a$, let $|a|=\sum_{v\in\mathcal N}\mathbb S_{va}$ and define the rescaled kinetic constant \begin{equation} \kappa_a=\frac{k_a}{\epsilon^{|a|-1}}, \label{eq:kappa-rescaled} \end{equation} where $\epsilon>0$ satisfies $\mathbf w\ge\epsilon\mathbf 1$. External species with fixed concentration are absorbed into the effective constants $\kappa_a$. The optimization model and solution method developed below depend only on the incidence matrices $\mathbb S$, $\mathbb T$, and the constants $\kappa_a$. They are therefore applicable to the whole class of scalable kinetics considered in \citet{gagrani2026topological}, independently of the reaction orders.

Under scalable dynamics, system~\eqref{eq:ode} admits
\emph{balanced-growth solutions}
$\mathbf{n}(t)=e^{\Lambda t}\mathbf{n}_0$,
with growth rate $\Lambda\in\mathbb{R}$ and composition
$\mathbf{n}_0>0$ satisfying
$\mathbb{Q}\mathbf{J}(\mathbf{n}_0)=\Lambda\mathbf{n}_0$.
Thus, balanced growth corresponds to an exponential growth regime in which all
species grow at the same asymptotic rate while preserving constant relative
concentrations. Such regimes play a central role in autocatalytic reaction systems because
they characterize the long-term behavior of self-sustaining reaction networks
capable of persistent exponential growth.
\begin{definition}[\cite{gagrani2026topological}]
\label{def:norma_cinetica}
Given a subhypergraph
$\mathcal{H}'=(\mathcal{N}',\mathcal{A}')$
with kinetic constants $\{\kappa_a\}_{a\in\mathcal{A}'}$, the
\emph{kinetic norm} of the source incidence matrix is
$$
\|\mathbb{S}|_{\mathcal{A}'}\|_\kappa
:=
\max_{v\in\mathcal{N}'}
\sum_{a\in\mathcal{A}'}
\mathbb{S}_{va}\kappa_a.
$$
\end{definition}
Here, $\kappa_a$ denotes the rescaled kinetic constant
in~\eqref{eq:kappa-rescaled}, with any fixed external concentrations absorbed
into it. For each node $v$, $\sum_{a\in\mathcal{A}'}
\mathbb{S}_{va}\kappa_a$ represents the total kinetic consumption rate of $v$ when all hyperarcs operate
at unit concentration. The maximum over all nodes therefore identifies the
largest kinetic consumption demand within the subhypergraph.

\begin{theorem}[\cite{gagrani2026topological}]
\label{thm:cotas}
Let $\mathcal{H}'=(\mathcal{N}',\mathcal{A}')$ be an autocatalytic
subhypergraph with respect to $\mathcal{M}=\mathcal{N}'$, endowed with the
linear-volume mass-action dynamics~\eqref{eq:linear-volume}, rescaled kinetic
constants $\{\kappa_a\}_{a\in\mathcal{A}'}$ given by~\eqref{eq:kappa-rescaled},
and Maximum Amplification Factor $\alpha^*$. Assume that the dynamics
admits a balanced-growth solution with strictly positive composition
$\mathbf{n}_0\in\mathbb{R}^{|\mathcal{N}'|}_{>0}$ and growth rate $\Lambda$.
Then
$$
-\|\mathbb{S}|_{\mathcal{A}'}\|_\kappa
<
\Lambda
\le
(\alpha^*-1)\,
\|\mathbb{S}|_{\mathcal{A}'}\|_\kappa.
$$
\end{theorem}
Both hypotheses are essential. The condition $\mathcal{M}=\mathcal{N}'$
ensures that the subnetwork is autonomous, as required
in~\citep{gagrani2026topological}. The positivity of $\mathbf{n}_0$ excludes
subnetworks in which some species vanish asymptotically. For instance, the
disjoint union of two autocatalytic motifs growing at different rates admits no
balanced-growth solution with $\mathbf{n}_0>0$.

The lower bound depends only on the kinetic constants and provides a universal
limit on the maximum contraction rate. The upper bound combines two
complementary mechanisms governing growth: the stoichiometric amplification,
measured by the factor $\alpha^*-1$, and the maximum kinetic consumption rate,
measured by $\|\mathbb{S}|_{\mathcal{A}'}\|_\kappa$. Consequently,
$(\alpha^*-1)\|\mathbb{S}|_{\mathcal{A}'}\|_\kappa$ provides an upper bound on
the balanced-growth rate. The optimization models developed in this paper aim
to identify autocatalytic subnetworks maximizing this upper bound.

To validate the bound, the realized balanced-growth rate $\Lambda$ must
be computed. For general reaction orders, this requires locating a positive
fixed point of the nonlinear composition dynamics
of~\citet{gagrani2026topological}. For this reason, our computational
validation focuses on first-order scalable kinetics, where each
hyperarc $a$ follows the interaction law
$J_a(\mathbf{n})=\kappa_a n_{s(a)}$, with $s(a)$ denoting the unique internal
source node of $a$. The corresponding hyperflow is
$$
\mathbf{J}(\mathbf{n})
=
\operatorname{diag}(\kappa)\,P\,\mathbf{n},
$$
where $P\in\{0,1\}^{|\mathcal{A}|\times|\mathcal{N}|}$ is the incidence matrix
identifying the source node of each hyperarc. Since each hyperarc has exactly
one internal source node, with unit stoichiometric coefficient, each row of $P$
contains exactly one entry equal to one, namely $P_{a,s(a)}=1$, and all other
entries are zero. Substituting this expression into
equation~\eqref{eq:ode} yields the linear dynamical system
\begin{equation}
\dot{\mathbf{n}}
=
M\mathbf{n},
\qquad
M=
\mathbb{Q}\operatorname{diag}(\kappa)P,
\label{eq:operador-M}
\end{equation}
where $\mathbb{Q}=\mathbb{T}-\mathbb{S}$ is the stoichiometric matrix.
Since each reaction has a unique internal reactant, the off-diagonal entries of $M$ are nonnegative.  Therefore, by
the Perron--Frobenius theorem for Metzler matrices, the eigenvalue of $M$ with
largest real part is real and admits a nonnegative eigenvector. If $M$ is irreducible, this eigenvector is strictly positive and the dominant eigenvalue coincides with the balanced-growth rate $\Lambda$ of Theorem~\ref{thm:cotas}. In all computational experiments, we verify this positivity condition.

For standard mass-action kinetics, the first-order subnetwork is precisely the linearization of the full system around the state with no internal species. Hence, its dominant eigenvalue $\Lambda$ describes the initial growth rate of small perturbations.
\section{The Growth-Bound Optimization Problem}\label{sec:problema}

Complex reaction networks typically contain many autocatalytic (or self-amplifying) 
subnetworks, each characterized by a different balance between
stoichiometric amplification and kinetic activity. Consequently, not all
autocatalytic subnetworks exhibit the same growth potential, and an
important question is how to identify those that are theoretically capable of
supporting the fastest balanced growth.

Theorem~\ref{thm:cotas} shows that the balanced-growth rate of a
autocatalytic subhypergraph is controlled by two complementary factors: the
Maximum Amplification Factor $\alpha^*$, which quantifies its stoichiometric
amplification capacity, and the kinetic norm
$\|\mathbb{S}|_{\mathcal{A}'}\|_{\kappa}$, which measures its maximum kinetic
consumption rate. Since both quantities depend on the selected subhypergraph,
maximizing the MAF alone does not necessarily maximize the upper bound on the
balanced-growth rate. Indeed, a subhypergraph may exhibit a large
amplification factor while remaining kinetically slow, leading to a weaker
growth bound than another subhypergraph with a smaller MAF but substantially
higher kinetic activity.

The optimization framework proposed in~\cite{gonzalez2024maf} identifies the
self-amplifying (or autocatalytic) subhypergraph with maximum amplification factor,
$$
\mathcal{H}^*
=
\arg\max_{\mathcal{H}'\subseteq\mathcal{H}}
\alpha^*(\mathcal{H}',\mathcal{M}).
$$
However, the previous discussion shows that maximizing $\alpha^*$ alone is
insufficient to maximize the growth-rate bound. This observation motivates the
main contribution of this work: instead of optimizing only the stoichiometric
amplification, we seek the autocatalytic subhypergraph maximizing the full
kinetic--stoichiometric upper bound,
\begin{equation}
\max_{\mathcal{H}'\subseteq\mathcal{H}}
\left(\alpha^*(\mathcal{H}',\mathcal{M})-1\right)
\|\mathbb{S}|_{\mathcal{A}'}\|_{\kappa}.
\label{eq:problema_principal}
\end{equation}
By Theorem~\ref{thm:cotas}, the optimal value of \eqref{eq:problema_principal} is an upper bound on the balanced-growth rate of any feasible autocatalytic subnetwork of $\mathcal H$. However, the subnetwork maximizing this bound need not maximize the realized growth rate.

To encode~\eqref{eq:problema_principal} as a mathematical optimization model, we introduce binary variables $z_a\in\{0,1\}$ indicating whether hyperarc $a\in\mathcal A$ belongs to $\mathcal A'$, binary variables $y_v\in\{0,1\}$ indicating whether node $v\in\mathcal N$ belongs to the amplifying set $\mathcal M$, and a hyperflow vector $\mathbf f\in\mathbb R_+^{|\mathcal A|}$. The support variables and hyperflows are linked by $z_a=\mathbf{1}[f_a>0]$ following~\citep{gonzalez2024maf}. With this encoding, the kinetic norm in Definition~\ref{def:norma_cinetica} becomes
$$
  \|\mathbb{S}|_{\mathcal A'}\|_{\kappa}
  =
  \max_{v\in\mathcal N}
  \sum_{a\in\mathcal A}\mathbb{S}_{va}\kappa_a z_a,
$$
that is, the maximum of $|\mathcal N|$ linear functions in the binary variables
$z_a$. Introducing an auxiliary variable $\mu\geq0$ and the constraints
\begin{equation}
  \mu \geq
  \sum_{a\in\mathcal A}\mathbb{S}_{va}\kappa_a z_a,
  \qquad \forall v\in\mathcal N,
  \label{eq:restriccion_mu}
\end{equation}
ensures that $\mu$ is at least as large as the kinetic consumption of every
node, i.e., $\mu\geq
\|\mathbb{S}|_{\mathcal A'}\|_{\kappa}$. However, these
inequalities alone do not enforce equality: $\mu$ could take an arbitrarily
larger value, making the objective unbounded since it is increasing in $\mu$.
To impose $\mu=\|\mathbb{S}|_{\mathcal A'}\|_{\kappa}$, we introduce
binary variables $w_v\in\{0,1\}$ selecting a node that attains the maximum, and
add
\begin{align}
  \sum_{a\in\mathcal A}\mathbb{S}_{va}\kappa_a z_a
  &\geq \mu w_v,
  && \forall v\in\mathcal N,
  \nonumber\\
  \sum_{v\in\mathcal N}w_v &\geq 1.
  \label{eq:mu_lower}
\end{align}
Thus, at least one node satisfies $w_v=1$ and must fulfill
$\sum_a\mathbb{S}_{va}\kappa_a z_a\geq\mu$. Together with
\eqref{eq:restriccion_mu}, this implies that this sum is exactly $\mu$, and
hence that $\mu$ coincides with the maximum over nodes.

Replacing the kinetic norm by $\mu$, the objective~\eqref{eq:problema_principal}
reduces to
\begin{equation}
  \max\;(\alpha^*-1)\mu,
  \label{eq:obj_bilineal}
\end{equation}
which is \emph{bilinear} in $(\alpha^*,\mu)$. Both factors depend on the same
binary selection variables $z_a$, and $\alpha^*$ is itself the optimal value of
the generalized fractional MAF problem in Definition~\ref{def:maf} over the
support $\{a:z_a=1\}$.

The constraints defining the feasible set are inherited from the MAF characterization in~\citep{gonzalez2024maf}, enriched with the coupling between the selected support $z_a$ and the kinetic norm. Let $\Delta$ be an upper bound on the hyperflows and let $\varepsilon=1/\Delta$ be the support threshold. Combining the hyperflow variables $\mathbf f$, the hyperarc-selection variables $\mathbf z$, the node-selection variables $\mathbf y$, and the kinetic-norm maximizer variables $\mathbf w$, the growth-bound optimization problem is written as:
\begin{align}
  \max_{\alpha,\mu,\mathbf f,\mathbf y,\mathbf z,\mathbf w}\quad
  &(\alpha-1)\mu
  \tag{P}\label{eq:problema_completo}\\
  \text{s.t.}\quad
  &\sum_{a\in\mathcal A}\mathbb{T}_{va}f_a
  \geq
  \alpha y_v\sum_{a\in\mathcal A}\mathbb{S}_{va}f_a,
  && \forall v\in\mathcal N,
  \label{eq:P-amp}\\[2pt]
  &y_v \leq
  \sum_{a\in\mathcal A:\mathbb{T}_{va}>0} f_a,\quad
  y_v \leq
  \sum_{a\in\mathcal A:\mathbb{S}_{va}>0} f_a,
  && \forall v\in\mathcal N,
  \label{eq:P-auto}\\[2pt]
  &\sum_{a\in\mathcal A}\mathbb{S}_{va}f_a \geq y_v,
  && \forall v\in\mathcal N,
  \label{eq:P-cons}\\[2pt]
  &f_a \leq
  \Delta\sum_{v\in\mathcal N:\mathbb{T}_{va}>0} y_v,\quad
  f_a \leq
  \Delta\sum_{v\in\mathcal N:\mathbb{S}_{va}>0} y_v,
  && \forall a\in\mathcal A,
  \label{eq:P-link}\\[2pt]
  & y_v \geq z_a, 
  && \forall a\in\mathcal A,\ v\in\mathcal N: \label{eq:P-autonomy}\\
  &&&\mathbb{S}_{va}+\mathbb{T}_{va}>0,
 \nonumber\\[2pt]
  &\varepsilon z_a \leq f_a \leq \Delta z_a,
  && \forall a\in\mathcal A,
  \label{eq:P-supp}\\[2pt]
  &\mu \geq
  \sum_{a\in\mathcal A}\mathbb{S}_{va}\kappa_a z_a,
  && \forall v\in\mathcal N,
  \label{eq:P-sup}
    \end{align} 
    \begin{align}
  &\sum_{a\in\mathcal A}\mathbb{S}_{va}\kappa_a z_a
  \geq \mu w_v,
  && \forall v\in\mathcal N,
  \label{eq:P-inf}\\[2pt]
  &\sum_{v\in\mathcal N} y_v \geq 1,\quad
  \sum_{a\in\mathcal A} f_a \geq 1,\quad
  \sum_{v\in\mathcal N} w_v \geq 1,
  \label{eq:P-norm}\\[2pt]
  &\alpha,\mu\geq 0,\quad
  \mathbf f\in\mathbb R_+^{|\mathcal A|},\quad
  \mathbf y,\mathbf w\in\{0,1\}^{|\mathcal N|},\quad
  \mathbf z\in\{0,1\}^{|\mathcal A|}.
  \nonumber
\end{align}

Constraint~\eqref{eq:P-amp} enforces self-amplification: for every 
node selected in the amplifying set ($y_v=1$), total production must be at least $\alpha$ times total
consumption, while the condition is inactive when $y_v=0$.
Constraints~\eqref{eq:P-auto} impose self-sufficiency, requiring every selected
node to be both produced and consumed by some positive-flow hyperarc.
Constraint~\eqref{eq:P-cons} ensures that each selected node is effectively
consumed, and~\eqref{eq:P-link} links the flow of each hyperarc to the selection
of its incident nodes. Constraint~\eqref{eq:P-autonomy} ensures that
$\mathcal{M}=\mathcal{N}'$, so that the selected subhypergraph is autonomous,
as required by Theorem~\ref{thm:cotas}. 
Constraint~\eqref{eq:P-supp} enforces
$z_a=1\iff f_a\geq\varepsilon$, so that the kinetic norm is evaluated exactly on
$\mathcal A'=\{a:z_a=1\}$. Finally, constraints~\eqref{eq:P-sup}
and~\eqref{eq:P-inf} jointly impose
$\mu=\|\mathbb{S}|_{\mathcal A'}\|_{\kappa}$: the former gives
$\mu\geq\|\mathbb{S}|_{\mathcal A'}\|_{\kappa}$, while the latter,
through $\sum_v w_v\geq1$, forces at least one node to attain that maximum.
Without~\eqref{eq:P-inf}, the objective would be unbounded because
\eqref{eq:P-sup} only provides a lower bound on $\mu$.

Let
$$
B(\mathcal{H}',\mathcal{M})
:=
\bigl(\alpha^*(\mathcal{H}',\mathcal{M})-1\bigr)
\,
\|\mathbb{S}|_{\mathcal{A}'}\|_{\kappa}
$$
denote the upper bound on the balanced growth rate provided by
Theorem~\ref{thm:cotas} for an autocatalytic subhypergraph
$(\mathcal{H}',\mathcal{M})$. The effectiveness of problem~(P) as a
decision criterion depends on the relative variability of the two factors
composing $B$ over the family $\mathcal{F}$ of feasible autocatalytic
subhypergraphs.

\begin{proposition}
\label{prop:reducciones}
Let $\mathcal{F}$ denote the family of feasible autocatalytic
subhypergraphs.

\begin{enumerate}
\item[(i)]
If $
\|\mathbb{S}|_{\mathcal{A}'}\|_{\kappa}
\equiv
\bar{\kappa}$, for all $(\mathcal{H}',\mathcal{M})\in\mathcal{F}$, 
then
$$
\arg\max_{(\mathcal{H}',\mathcal{M})\in\mathcal{F}}
B(\mathcal{H}',\mathcal{M})
=
\arg\max_{(\mathcal{H}',\mathcal{M})\in\mathcal{F}}
\alpha^*(\mathcal{H}',\mathcal{M}).
$$
Hence, maximizing the growth bound is equivalent to maximizing the MAF.

\item[(ii)]
If $\alpha^*(\mathcal{H}',\mathcal{M})
\equiv
\bar{\alpha}$, for all $(\mathcal{H}',\mathcal{M})\in\mathcal{F}$, 
then
$$
\arg\max_{(\mathcal{H}',\mathcal{M})\in\mathcal{F}}
B(\mathcal{H}',\mathcal{M})
=
\arg\max_{(\mathcal{H}',\mathcal{M})\in\mathcal{F}}
\|\mathbb{S}|_{\mathcal{A}'}\|_{\kappa}.
$$
Consequently, the optimization criterion ranks feasible subhypergraphs
exclusively according to their kinetic consumption norm.
\end{enumerate}

Both statements follow immediately from the multiplicative structure of
$B(\mathcal{H}',\mathcal{M})$.
\end{proposition}

\begin{lemma}
\label{lem:ajuste}
Consider the elementary autocatalytic motif $A\rightarrow mA$, with $m \in \mathbb Z$, $m\geq 2$, and
with kinetic constant $\kappa$. In isolation,
$$
\alpha^*=m,
\qquad
\|\mathbb{S}\|_{\kappa}=\kappa,
\qquad
\Lambda=(m-1)\kappa.
$$
Therefore, $\Lambda
=
(\alpha^*-1)
\|\mathbb{S}\|_{\kappa}$, 
and the upper bound of Theorem~\ref{thm:cotas} is attained with equality.
\end{lemma}

To quantify the tightness of the bound, define the ratio
$$
\tau(\mathcal{H}',\mathcal{M})
:=
\frac{\Lambda(\mathcal{H}',\mathcal{M})}
{B(\mathcal{H}',\mathcal{M})}
\in(0,1].
$$
The optimization criterion induced by problem~(P) provides a reliable
ranking whenever the variability of $\tau$ across feasible
subhypergraphs is small compared with the variability of
$B(\mathcal{H}',\mathcal{M})$. Lemma~\ref{lem:ajuste} establishes that
$\tau=1$ for single-node autocatalytic motifs, where the node attaining
the kinetic norm is precisely the growth-limiting node. As this
correspondence deteriorates, the bound becomes progressively less tight.

\begin{observation}
\label{obs:fallo}
The growth-bound criterion may fail to correctly rank feasible
subhypergraphs when the following two conditions simultaneously hold:
\begin{enumerate}
\item[(i)]
the amplification factors $\alpha^*$ are nearly constant over the family
$\mathcal{F}$, so that Proposition~\ref{prop:reducciones}(ii) implies
that the optimization is driven almost entirely by the kinetic norm,
\item[(ii)]
the node attaining
$\|\mathbb{S}|_{\mathcal{A}'}\|_{\kappa}$
does not belong to the rate-limiting growth cycle determining
$\Lambda$.
\end{enumerate}

Under these conditions, the kinetic norm ceases to be a reliable proxy
for the balanced growth rate, and maximizing the upper bound
$B(\mathcal{H}',\mathcal{M})$ may not identify the subhypergraph with the
largest realized growth rate.
\end{observation}
 
\section{Exact Solution Methodology}\label{sec:sweep}

The optimization problem~\eqref{eq:problema_completo} is challenging because
the kinetic norm, the amplification factor, and the hypergraph selection are
simultaneously optimized. This results in a nonconvex mixed-integer program
containing bilinear and trilinear terms. To solve it exactly, we exploit the
fact that the kinetic norm can attain only finitely many values and decompose
the problem into a family of parametric subproblems. For a fixed value $\mu=\mu_j$ of the kinetic norm, the objective becomes
equivalent to maximizing the amplification factor over all autocatalytic
subhypergraphs having kinetic norm $\mu_j$. This yields the parametric problem
\begin{equation}
  P(\mu_j) :=
  \max_{\mathcal H'\subseteq\mathcal H}\;\alpha^*(\mathcal H',\mathcal M)
  \quad\text{s.t.}\quad
  \|\mathbb S|_{\mathcal A'}\|_{\kappa}=\mu_j .
  \label{eq:parametrico}
\end{equation}
The global optimum of~\eqref{eq:problema_completo} is recovered by evaluating the bound at every attainable value $\mu_1<\cdots<\mu_N$ of the kinetic norm and retaining the best value:
\begin{equation}
  \max_{\mathcal H'\subseteq\mathcal H}
  \bigl(\alpha^*(\mathcal H',\mathcal M)-1\bigr)
  \|\mathbb S|_{\mathcal A'}\|_{\kappa}
  =
  \max_{1\le j\le N}
  \bigl(P(\mu_j)-1\bigr)\mu_j .
  \label{eq:barrido-optimo}
\end{equation}

Fixing $\mu=\mu_j$ eliminates every bilinearity involving the kinetic norm.
The objective becomes linear in $\alpha$, and the kinetic constraints become
linear in the support variables. The only remaining source of nonconvexity is
the self-amplification constraint~\eqref{eq:P-amp}. This constraint is handled
using the generalized fractional programming approach introduced
in~\cite{crouzeix1985,gonzalez2024maf}. For a
trial value $\alpha_{\texttt{it}}$, constraint~\eqref{eq:P-amp} is linearized
and the joint selection of a subhypergraph and a hyperflow is handled through
the following auxiliary max--min MILP:
\begin{align}
  A(\alpha_{\texttt{it}},\mu_j):\quad
  \max_{\rho,\mathbf f,\mathbf y,\mathbf z,\mathbf w}\quad
  & \rho
    \label{eq:milp-aux}\\
  \text{s.t.}\quad
  & \rho \le
    (\mathbb T\mathbf f)_v
    -\alpha_{\texttt{it}}(\mathbb S\mathbf f)_v
    +M(1-y_v),
    && \forall v\in\mathcal N,
    \label{eq:aux-maximin}\\[1pt]
  & y_v \le
    \sum_{a\in\mathcal A:\mathbb T_{va}>0} f_a,\quad
    y_v \le
    \sum_{a\in\mathcal A:\mathbb S_{va}>0} f_a,
    && \forall v\in\mathcal N,
    \nonumber\\[1pt]
  & (\mathbb S\mathbf f)_v \ge y_v,
    && \forall v\in\mathcal N,
    \nonumber\\[1pt]
  & f_a \le
    \Delta\sum_{v\in\mathcal N:\mathbb T_{va}>0} y_v,\quad
    f_a \le
    \Delta\sum_{v\in\mathcal N:\mathbb S_{va}>0} y_v,
    && \forall a\in\mathcal A,
    \nonumber\\[1pt]
  & y_v \ge z_a, 
    && \forall a\in\mathcal A,\ v\in\mathcal N: \nonumber \\ &\qquad &&\mathbb S_{va}+\mathbb T_{va}>0,
    \nonumber\\[1pt]
  & \varepsilon z_a \le f_a \le \Delta z_a,
    && \forall a\in\mathcal A,
    \nonumber
    \end{align}
    \begin{align}
  & \mu_j \ge
    \sum_{a\in\mathcal A}\mathbb S_{va}\kappa_a z_a,
    && \forall v\in\mathcal N,
    \nonumber\\[1pt]
  & \sum_{a\in\mathcal A}\mathbb S_{va}\kappa_a z_a
    \ge \mu_j w_v,
    && \forall v\in\mathcal N,
    \nonumber\\[1pt]
  & \sum_{v\in\mathcal N} y_v \ge 1,\quad
    \sum_{a\in\mathcal A} f_a \ge 1,\quad
    \sum_{v\in\mathcal N} w_v \ge 1,
    \nonumber\\[1pt]
  & \mathbf f\in\mathbb R_+^{|\mathcal A|},\quad
    \mathbf y,\mathbf w\in\{0,1\}^{|\mathcal N|},\quad
    \mathbf z\in\{0,1\}^{|\mathcal A|},\quad
    \rho\in\mathbb R.
    \nonumber
\end{align}
Here $(\mathbb T\mathbf f)_v=\sum_a\mathbb T_{va}f_a$ and
$(\mathbb S\mathbf f)_v=\sum_a\mathbb S_{va}f_a$ denote, respectively, the
production and consumption of node $v$. The constant
$$
M=
\Delta\left(
\max_v\sum_a\mathbb T_{va}
+
\alpha_{\rm it}
\max_v\sum_a\mathbb S_{va}
\right)
$$
is sufficiently large to deactivate
constraint~\eqref{eq:aux-maximin} whenever $y_v=0$. Consequently,
$\rho$ equals the minimum self-amplification margin over the selected nodes,
that is,
$$
\rho=
\min_{v:y_v=1}
\left[
(\mathbb T\mathbf f)_v
-
\alpha_{\rm it}
(\mathbb S\mathbf f)_v
\right].
$$
Hence, maximizing $\rho$ searches for the hyperflow providing the largest
minimum amplification margin for the current trial value
$\alpha_{\rm it}$. If $\rho^*=0$, the current value of
$\alpha_{\rm it}$ is optimal for the fixed kinetic norm.
Otherwise, $\rho^*>0$ certifies that a larger amplification factor exists, and
the trial value is updated according to the generalized fractional programming
scheme.

The efficiency of the proposed solution method relies on the fact that the
parameter $\mu$ does not vary continuously. Since the support vector
$\mathbf z$ is binary, the kinetic norm can attain only finitely many values.
Therefore, the outer parametric sweep reduces to solving a finite family of
subproblems. The following theorem characterizes the finite set of attainable kinetic norms.
\begin{theorem}
\label{thm:mu_values}
Let $\mathcal Z=
\left\{
\mathbf z\in\{0,1\}^{|\mathcal A|}:
\exists\,(\alpha,\mathbf f,\mathbf y,\mathbf w)
\text{ satisfying }
\eqref{eq:P-amp}\text{--}\eqref{eq:P-norm}
\right\}$, 
and define $
\mu(\mathbf z)=
\max_{v\in\mathcal N}
\sum_{a\in\mathcal A}
\mathbb S_{va}\kappa_a z_a$, for all
$\mathbf z\in\mathcal Z$. 
Then the attainable kinetic norms are precisely
$$
\mathcal U=
\{\mu(\mathbf z):\mathbf z\in\mathcal Z\}
=
\{\mu_1<\cdots<\mu_N\}.
$$
Moreover,
$$
0\le
\mu_j
\le
\bar\mu:=
\max_{v\in\mathcal N}
\sum_{a\in\mathcal A}
\mathbb S_{va}\kappa_a,
$$
and
$$
N\le |\mathcal Z|
\le
2^{|\mathcal A|}.
$$
Consequently, the outer parametric sweep only needs to solve
$P(\mu_j)$ for
$\mu_j\in\mathcal U$.
\end{theorem}
\begin{proof}
For every feasible support vector $\mathbf z\in\mathcal Z$, the value
$\mu(\mathbf z)$ is uniquely determined by
$$
\mu(\mathbf z)=
\max_{v\in\mathcal N}
\sum_{a\in\mathcal A}
\mathbb S_{va}\kappa_a z_a.
$$
Hence, $\mathcal U$ is finite because $\mathcal Z\subseteq\{0,1\}^{|\mathcal A|}$. Therefore,
$N=|\mathcal U|
\le
|\mathcal Z|
\le
2^{|\mathcal A|}$. 
Moreover, since $0\le z_a\le1$ for every $a\in\mathcal A$, we get that $
0
\le
\mu(\mathbf z)
\le
\max_{v\in\mathcal N}
\sum_{a\in\mathcal A}
\mathbb S_{va}\kappa_a
=
\bar\mu$.
\end{proof}

Theorem~\ref{thm:mu_values} characterizes the finite set of attainable
kinetic norms at the conceptual level. In practice, the attainable kinetic norms are explored after converting the
kinetic constants into scaled integers. Each constant is represented as
$\kappa_a\approx k_a^{\rm int}/s$, where $k_a^{\rm int}$ is integer and $s$ is
a common scaling factor. Consequently, every attainable kinetic norm becomes
an integer multiple of $1/s$, so that the outer sweep can be performed over
integer norm levels. The exactness guarantee of the sweep applies to the scaled
instance defined by $k^{\rm int}$: if the original kinetic constants are
represented exactly by this scaling, the scaled and original instances
coincide, if they are rounded before scaling, the selected support is evaluated
afterwards using the original kinetic constants, while the optimality
certificate of Proposition~\ref{prop:exactitud} refers to the scaled instance.

Although Theorem~\ref{thm:mu_values} establishes that the set of attainable kinetic norms is finite, it does not identify its elements explicitly. Enumerating all integer levels is inefficient, since many do not correspond to any feasible subhypergraph. Instead, given an upper bound $\bar\mu$, the next attainable norm is obtained by solving
\begin{align}
N(\bar\mu):\quad \max\ \ & \mu \nonumber\\
\text{s.t. }\ & \text{\eqref{eq:P-auto}--\eqref{eq:P-autonomy} on }
(\mathbf y,\mathbf z),\nonumber\\
& \sum_{a\in\mathcal A}\mathbb S_{va}\kappa_a z_a\ \le\ \mu,
&& \forall v\in\mathcal N,\nonumber\\
& \sum_{v\in\mathcal N} w_v\ \ge\ 1,\nonumber\\
& \sum_{a\in\mathcal A}\mathbb S_{va}\kappa_a z_a\ \ge\
\mu-\bar\mu\,(1-w_v),
&& \forall v\in\mathcal N,\nonumber\\
& \mu\le\bar\mu,\quad \mu\in\mathbb Z_+,\ \
(\mathbf y,\mathbf z,\mathbf w)\in
\{0,1\}^{|\mathcal N|+|\mathcal A|+|\mathcal N|}.
\label{eq:next-norm}
\end{align}
The constraints \eqref{eq:P-auto}--\eqref{eq:P-autonomy} enforce the structural feasibility of the support. The second family ensures that $\mu$ is at least the kinetic consumption of every node, while the third and fourth families force at least one node to attain this value. Consequently, $\mu$ is exactly the kinetic norm of the selected support. The constraint $\mu\le\bar\mu$ restricts the search to norms not exceeding the current level.

\begin{lemma} \label{lem:next-norm}Problem~\eqref{eq:next-norm} is infeasible if and only if no attainable kinetic norm satisfies $\mu\le\bar\mu$. Otherwise, its optimal value is 
\begin{align*}
\max\{\mu\in\mathcal U:\mu\le\bar\mu\}. 
\end{align*}
\end{lemma}

\begin{proof}
A support vector $\mathbf z$ is feasible for~\eqref{eq:next-norm} if and only if it satisfies the structural constraints of the growth-bound model and induces a kinetic norm not exceeding $\bar\mu$. The feasible values of $\mu$ are therefore precisely
\begin{align*}
    \{\mu\in\mathcal U:\mu\le\bar\mu\},
\end{align*}
and maximizing $\mu$ returns the largest element of this set whenever it is nonempty. 
\end{proof}
By Lemma~\ref{lem:next-norm}, repeatedly solving \eqref{eq:next-norm} and updating $\bar\mu$ to $\mu_j-1$ enumerates the attainable kinetic norms in decreasing order. Consequently, every parametric subproblem processed by the sweep corresponds to an attainable norm and is therefore feasible. Since~\eqref{eq:next-norm} only involves the structural variables, it is considerably simpler than the amplification subproblems.

The descending sweep naturally admits an incumbent-based screening strategy. After solving the first $j$
parametric subproblems, let
$$
  \mathrm{LB}_j=\max_{\ell\le j}(\alpha_\ell^*-1)\mu_\ell
$$
be the best objective value found so far. Let $\alpha_{\rm glob}$ denote the
global maximum amplification factor, obtained by maximizing $\alpha^*$ over all
subhypergraphs without fixing the kinetic norm. Since $\alpha^*(\mu)\le
\alpha_{\rm glob}$ for every attainable norm $\mu$, the value $\alpha_{\rm
glob}$ is a valid constant upper bound on every parametric subproblem, and any
remaining norm $\mu$ satisfying $(\alpha_{\rm glob}-1)\,\mu\le\mathrm{LB}_j$ 
cannot improve the incumbent and can be discarded without solving $P(\mu)$.
Because $(\alpha_{\rm glob}-1)\mu$ is increasing in $\mu$ and the grid is
traversed in decreasing order, once this test holds it holds for all smaller
norms as well, so the outer sweep can terminate. This incumbent-based screening
substantially reduces the number of parametric subproblems while preserving
global optimality.

The bound $\alpha_{\rm glob}$ is itself computed with the same
Dinkelbach-type scheme, and it is a valid upper bound only once certified as
optimal. When the auxiliary problems defining it are solved to proven
optimality, screening is enabled. Otherwise it is disabled and the full scaled
grid is explored, which preserves exactness at the cost of solving more
subproblems.

Algorithm~\ref{alg:3} summarizes the complete solution procedure. The kinetic constants are first converted to the integer vector
$k^{\rm int}$. In the scaled units all kinetic consumption values are integer,
and the outer loop traverses the corresponding grid in decreasing order. For
each norm level, the incumbent-based test determines whether the corresponding
parametric subproblem must be solved, and this test is active only when
$\alpha_{\rm glob}$ has been certified.

\begin{algorithm}[h!]
{\small
\caption{Descending scaled-norm sweep with certified incumbent screening.
\label{alg:3}}

\SetKwInOut{Input}{input}
\SetKwInOut{Output}{output}

\Input{$\mathcal H=(\mathcal N,\mathcal A)$, kinetic constants
$\{\kappa_a\}_{a\in\mathcal A}$, decimal precision $d$.}

\Output{Selected subhypergraph $\mathcal H'^*$, amplification factor
$\alpha^{**}$, kinetic norm $\mu^*$, and objective value $\mathrm{LB}$.}

Convert $\kappa$ into scaled integer coefficients $k^{\rm int}$\;
$\bar\mu\leftarrow\max_{v\in\mathcal N}\sum_{a\in\mathcal A}
\mathbb S_{va}k^{\rm int}_a$

Compute the global MAF $\alpha_{\rm glob}$ without fixing the norm\;

\If{$\alpha_{\rm glob}$ is certified globally optimal}{
    $\texttt{screening}\leftarrow\texttt{true}$\;
}
\Else{
    $\texttt{screening}\leftarrow\texttt{false}$\;
}

$\mathrm{LB}\leftarrow-\infty$\;

\While{$N(\bar\mu)$ is feasible}{

    $\mu_j\leftarrow$ optimal value of $N(\bar\mu)$\;

    \If{$\texttt{screening}$ and
    $(\alpha_{\rm glob}-1)\mu_j\le\mathrm{LB}$}{
        \textbf{break}\;
    }

    Solve $A(0,\mu_j)$

    Compute $\alpha_{\rm it}$ from the returned hyperflow\;

    \Repeat{$|\rho^*|\le\eta$}{

        Solve $A(\alpha_{\rm it},\mu_j)$ and obtain
        $(\rho^*,\mathbf f,\mathbf y,\mathbf z,\mathbf w)$\;

        \If{$\rho^*>\eta$}{
            $\displaystyle
            \alpha_{\rm it}\leftarrow
            \min_{v:y_v=1}
            \frac{(\mathbb T\mathbf f)_v}
                 {(\mathbb S\mathbf f)_v}$\;
        }
    }

    Recompute the kinetic norm $\mu_{\rm real}$ using the original
    kinetic constants on the selected support\;

    $\mathrm{obj}\leftarrow(\alpha_{\rm it}-1)\mu_{\rm real}$\;

    \If{$\mathrm{obj}>\mathrm{LB}$}{
        $\mathrm{LB}\leftarrow\mathrm{obj}$\;
        $\alpha^{**}\leftarrow\alpha_{\rm it}$\;
        $\mu^*\leftarrow\mu_{\rm real}$\;
        $\mathcal H'^*\leftarrow
        (\{v:y_v=1\},\{a:z_a=1\})$\;
    }

    $\bar\mu\leftarrow\mu_j-1$\;
}

\Return{$\mu^*,\mathcal H'^*,\alpha^{**},\mathrm{LB}$}\;

}
\end{algorithm}

Finite convergence of the inner generalized fractional programming procedure
follows directly from the analysis in~\cite{gonzalez2024maf}. Let $\rho^*(\alpha)$ denote the optimal
value of the auxiliary problem~\eqref{eq:milp-aux} as a function of the trial
value $\alpha$. At each nonterminal iteration, the trial value
$\alpha_{\mathrm{it}}$ increases strictly. Since $\rho^*(\alpha)$ is
nonincreasing in $\alpha$ and vanishes at $\alpha=P(\mu_j)$, the sequence
$\{\alpha_{\mathrm{it}}\}$ is monotone increasing and bounded above by
$P(\mu_j)$. Moreover, the auxiliary MILP
\eqref{eq:milp-aux} admits only finitely many integer configurations
$(\mathbf y,\mathbf z,\mathbf w)$, implying that the inner procedure converges
after finitely many iterations when the auxiliary MILPs are solved to proven
optimality. In numerical computations, the stopping condition
$\rho^*=0$ is implemented with a tolerance $|\rho^*|\le\eta$.

Together with the finiteness of the outer search
(Theorem~\ref{thm:mu_values}), this yields the global optimality of the overall
procedure, which we now state formally.

\begin{proposition}
\label{prop:exactitud}
Assume that the kinetic constants are represented exactly in the scaled integer
vector $k^{\rm int}$, or that the rounded constants define the target
optimization instance. Assume also that the auxiliary
problems~\eqref{eq:next-norm}, the global MAF used for
screening, and every auxiliary MILP are
solved to proven optimality, and that the global MAF is certified as a
valid upper bound. Then Algorithm~\ref{alg:3} terminates after
finitely many iterations and returns a globally optimal solution for the
corresponding scaled instance.
\end{proposition}
\begin{proof}
By Theorem~\ref{thm:mu_values} and the integer scaling, the set $\mathcal U$ of
attainable norms is finite, and by Lemma~\ref{lem:next-norm} the outer loop
visits its elements in decreasing order and terminates once they are
exhausted. For each of them the inner
Dinkelbach-type loop terminates finitely (as shown above). It thus remains to
check that the sweep returns the optimum. Every feasible support of the scaled
instance attains exactly one norm level $\mu_j\in\mathcal U$, so
$\max_{\mu_j\in\mathcal U}(P(\mu_j)-1)\mu_j$ is the optimal value of the
growth-bound problem. When screening is disabled this value is computed
directly. When it is enabled, a level $\mu_j$ is skipped only if
$(\alpha_{\rm glob}-1)\mu_j\le\mathrm{LB}$. Since $\alpha_{\rm glob}$ is a
certified upper bound on $\alpha^*(\mu_j)$, every solution at that level scores
at most $\mathrm{LB}$ and cannot be optimal. As $(\alpha_{\rm glob}-1)\mu_j$
increases in $\mu_j$ and the levels are visited in decreasing order, this test, once satisfied,
holds for all smaller levels, so the skipped levels are exactly a dominated
tail. The incumbent therefore coincides with the maximum above.
\end{proof}

\section{Analytical examples}\label{sec:examples}

Before presenting the computational results, we examine two illustrative
examples that clarify the behavior of the proposed framework. The first shows
that maximizing the Maximum Amplification Factor (MAF) and maximizing the
kinetic--stoichiometric growth bound may favor different autocatalytic
subhypergraphs. The second provides an analytical verification on the
\emph{Oregonator}, a classical chemical reaction network for which the optimal
amplification factor can be derived in closed form.

\subsection{MAF maximization versus growth-bound maximization}
\label{sec:tesis}

The framework introduced in~\cite{gonzalez2024maf} identifies
autocatalytic subhypergraphs with maximum amplification factor $\alpha^*$.
By contrast, Algorithm~\ref{alg:3} maximizes the complete
kinetic--stoichiometric upper bound
$(\alpha^*-1)\|\mathbb S|_{\mathcal A'}\|_{\kappa}$. These two criteria are
not equivalent: a subhypergraph with a large stoichiometric amplification may
have slow kinetics and therefore a smaller growth bound than a less amplifying
but kinetically faster subhypergraph.

To illustrate this distinction, consider the two inclusion-minimal
autocatalytic components of the following disconnected CRN:
$$
\text{Block A:}\qquad
A\xrightarrow{\kappa_A}3A,
\qquad\qquad
\text{Block B:}\qquad
B\xrightarrow{\kappa_B}2B,
$$
where $\kappa_A=0.3$ and $\kappa_B=1$. The comparison is restricted to
inclusion-minimal autocatalytic subhypergraphs. Block A exhibits stronger
stoichiometric amplification, whereas Block B operates on a faster kinetic
scale.

Since each block consists of a single autocatalytic hyperarc, all relevant
quantities admit closed-form expressions. For a reaction
$A\rightarrow mA$ with kinetic constant $\kappa$, one has
$$
\alpha^*=m,\qquad
\|\mathbb S|_{\mathcal A'}\|_{\kappa}=\kappa,\qquad
(\alpha^*-1)\|\mathbb S|_{\mathcal A'}\|_{\kappa}=(m-1)\kappa.
$$
Table~\ref{tab:tesis} summarizes the corresponding values. The MAF criterion
ranks Block A above Block B because it has the larger amplification factor,
whereas the growth-bound criterion ranks Block B above Block A because its
faster kinetics more than compensate for its smaller stoichiometric
amplification.

\begin{table}[h]
\centering
\begin{tabular}{lccccc}
\toprule
Subhypergraph &
$\alpha^*$ &
$\|\mathbb S|_{\mathcal A'}\|_{\kappa}$ &
$(\alpha^*-1)\|\mathbb S|_{\mathcal A'}\|_{\kappa}$ &
$\Lambda$ &
Selected by\\
\midrule
Block A ($A\rightarrow3A$)
&
3
&
0.30
&
0.60
&
0.60
&
MAF\\
Block B ($B\rightarrow2B$)
&
2
&
1.00
&
1.00
&
1.00
&
Algorithm~\ref{alg:3}\\
\bottomrule
\end{tabular}
\caption{Designed instance illustrating the difference between maximizing the
MAF and maximizing the proposed growth bound.}
\label{tab:tesis}
\end{table}

For these one-species systems, the balanced-growth rate satisfies
$\Lambda=(m-1)\kappa$ and therefore coincides with the upper bound.
Consequently, Block B has the larger balanced-growth rate despite its smaller
amplification factor. In this setting, the example shows that the MAF alone
does not correctly rank the components according to their realized growth,
whereas the proposed bound does.

The optimum of problem~\eqref{eq:problema_completo} is degenerate for this instance. The union of both blocks attains the same objective value and is in fact the solution returned by Algorithm~\ref{alg:3}. However, its growth operator is reducible and does not satisfy the assumptions of Theorem~\ref{thm:cotas}. The effective growth is therefore determined by Block~B, which is the component reported in Table~\ref{tab:tesis}.

To examine the trade-off between amplification and kinetics, fix
$\kappa_A=0.3$ and replace $\kappa_B$ by $\lambda\kappa_B$. The growth-bound
values of Blocks A and B are then $0.6$ and $\lambda$, respectively.
Therefore, the preferred component switches from Block A to Block B at
$\lambda=0.6$, as shown in Figure~\ref{fig:tesis_switch}. This transition
illustrates how the proposed objective balances stoichiometric amplification
against kinetic speed.
\begin{figure}[h]
\centering
\includegraphics[width=0.62\textwidth, trim={0.45cm 0.5cm 0.35cm 0.35cm},
    clip]{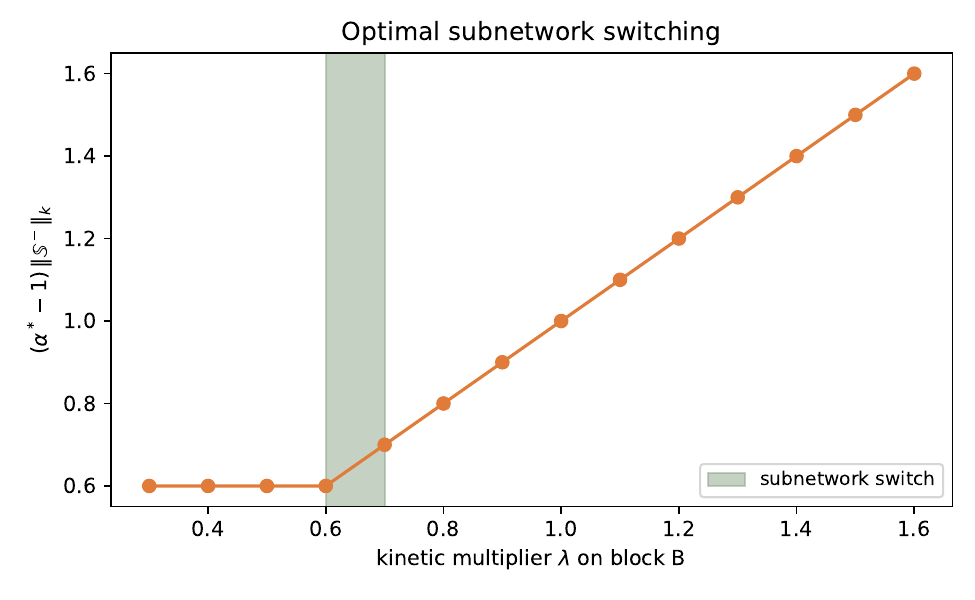}
\caption{Growth-bound values of Blocks A and B when the kinetic constant of
Block B is scaled by $\lambda$. The two criteria coincide at
$\lambda=0.6$, where the preferred component switches from Block A to
Block B.}
\label{fig:tesis_switch}
\end{figure}

\subsection{Analytical verification on the Oregonator}
\label{sec:oregonator}

We next consider the Oregonator~\citep{field1974oscillations}, a classical
kinetic model of the Belousov--Zhabotinsky oscillatory reaction. This example
provides a chemically meaningful benchmark because its reduced
autocatalytic core has an amplification factor that can be derived
analytically and compared directly with the value returned by
Algorithm~\ref{alg:3}.

The original model consists of five reactions involving three internal species,
namely bromous acid ($X$), bromide ion ($Y$), and cerium ion ($Z$), together
with two external species, bromate ($A$, $[A]=0.06$ M) and malonic acid
($B$, $[B]=0.10$ M):
\begin{align*}
  \mathrm R_1:\;&A+Y\rightarrow X+P
  &&(\kappa_1=1.28),\\
  \mathrm R_2:\;&X+Y\rightarrow2P
  &&(\kappa_2=2.4\times10^6),\\
  \mathrm R_3:\;&A+X\rightarrow2X+2Z
  &&(\kappa_3=33.6),\\
  \mathrm R_4:\;&2X\rightarrow A+P
  &&(\kappa_4=3\times10^3),\\
  \mathrm R_5:\;&B+Z\rightarrow Y
  &&(\kappa_5=0.5),
\end{align*}
where the reported bimolecular kinetic constants are expressed in
$\mathrm{M}^{-1}\mathrm{s}^{-1}$.

Following common practice, the stoichiometric factor of $\mathrm R_5$ is set to $f=1$. Reactions $\mathrm R_2$ and $\mathrm R_4$ do not produce any internal species and therefore cannot belong to an autocatalytic subhypergraph by Definition~\ref{def:autoamplificante}. Consequently, the relevant autocatalytic core consists of reactions $\mathrm R_1$, $\mathrm R_3$, and $\mathrm R_5$. 

Moreover, since $\mathrm R_2$ and $\mathrm R_4$ are quadratic in the internal concentrations, the reduced system coincides with the linearization of the full Oregonator around $X=Y=Z=0$. Hence, the balanced-growth rate $\Lambda$ describes the initial growth of small perturbations from this state.  Treating the external species as
buffered and absorbing their fixed concentrations into the corresponding
rate constants gives the effective first-order constants
$\kappa_1^{\mathrm{eff}}=1.28[A]=0.0768\,\mathrm{s}^{-1}$,
$\kappa_3^{\mathrm{eff}}=33.6[A]=2.016\,\mathrm{s}^{-1}$, and
$\kappa_5^{\mathrm{eff}}=0.5[B]=0.05\,\mathrm{s}^{-1}$.

The resulting autocatalytic core is represented by the incidence matrices
(rows correspond to $X$, $Y$, and $Z$, while columns correspond to
$\mathrm R_1$, $\mathrm R_3$, and $\mathrm R_5$):
$$
\mathbb S=
\begin{pmatrix}
0&1&0\\
1&0&0\\
0&0&1
\end{pmatrix},
\qquad
\mathbb T=
\begin{pmatrix}
1&2&0\\
0&0&1\\
0&2&0
\end{pmatrix},
\qquad
\boldsymbol{\kappa}=
\begin{pmatrix}
0.0768\\
2.016\\
0.05
\end{pmatrix}.
$$
The reduced network contains the directly autocatalytic reaction
$\mathrm R_3$, which converts one molecule of $X$ into two molecules of $X$
and two molecules of $Z$. When reactions $\mathrm R_1$ and $\mathrm R_5$ are
also selected, the resulting cycle
$\{\mathrm R_1,\mathrm R_3,\mathrm R_5\}$ introduces the additional
regeneration pathway
$\mathrm R_3\rightarrow Z\rightarrow Y\rightarrow X$.
As shown below, this indirect pathway increases the amplification factor beyond
the value $\alpha^*=2$ attained by the isolated reaction $\mathrm R_3$.

\begin{figure}[h]
\centering
\includegraphics[width=0.5\textwidth, trim={0.7cm 0.4cm 2.2cm 0.9cm},
    clip]{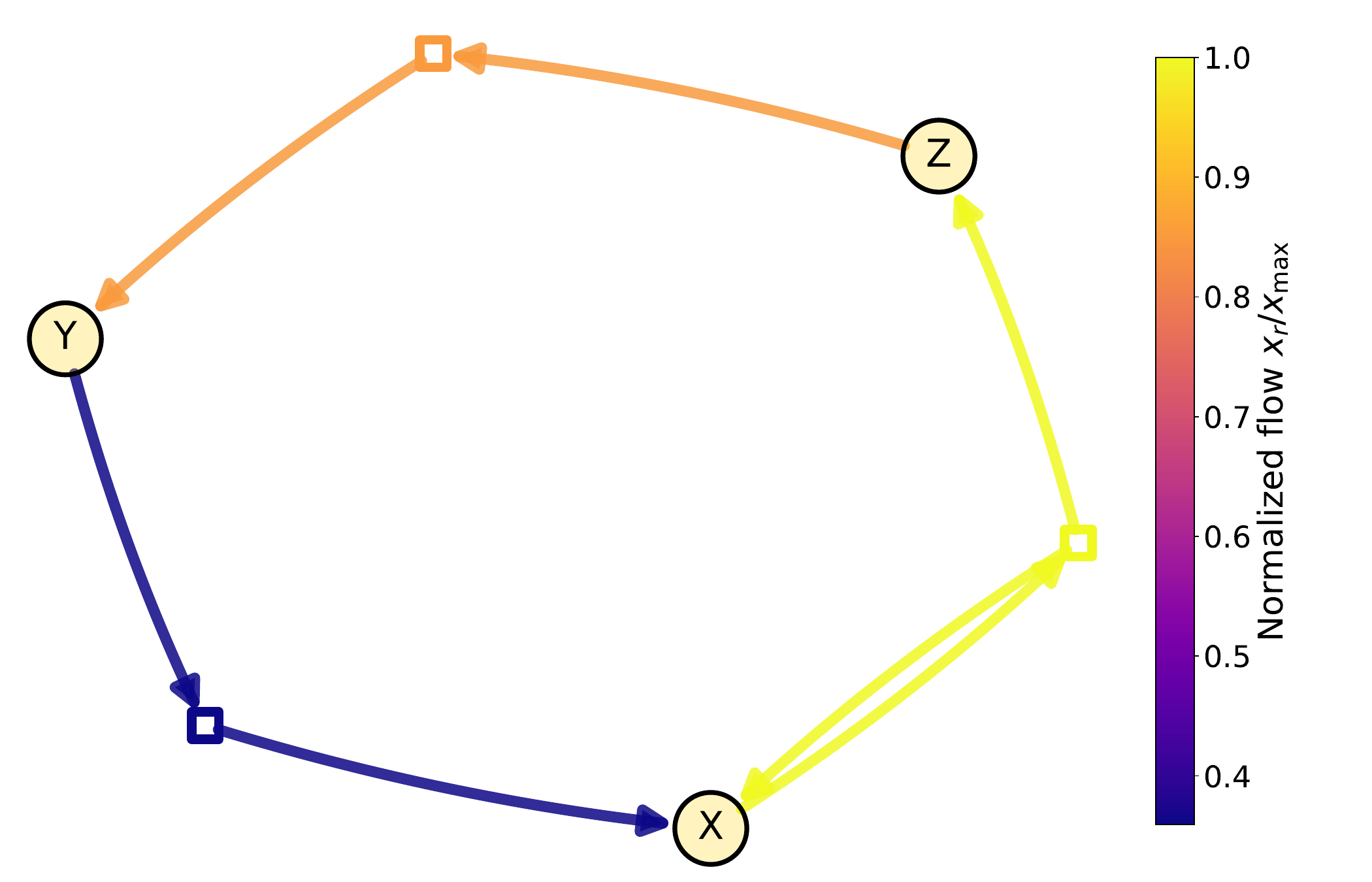}
\caption{autocatalytic core of the Oregonator and autocatalytic cycle
$\{\mathrm R_1,\mathrm R_3,\mathrm R_5\}$ recovered by
Algorithm~\ref{alg:3}. Reaction $\mathrm R_3$ regenerates species $X$
directly, while the pathway
$\mathrm R_3\!\rightarrow\!Z\!\rightarrow\!Y\!\rightarrow\!X$
provides an indirect regeneration mechanism that increases the overall
amplification factor above $2$.}
\label{fig:oregonator_subnet}
\end{figure}
The MAF of the complete cycle can be obtained directly from
Definition~\ref{def:maf}. Let $f_1$, $f_3$, and $f_5$ denote the hyperflows
through reactions $\mathrm R_1$, $\mathrm R_3$, and $\mathrm R_5$,
respectively. At an optimal strictly positive hyperflow, the
production-to-consumption ratios of the three internal species coincide:
$$
\frac{f_1+2f_3}{f_3}
=
\frac{f_5}{f_1}
=
\frac{2f_3}{f_5}
=
\alpha.
$$
Eliminating $f_1$, $f_3$, and $f_5$ yields
\begin{equation}
\alpha^3-2\alpha^2-2=0,
\label{eq:cubica_oregonator}
\end{equation}
whose unique real root is
$\alpha^*=2.359304\ldots$. Algorithm~\ref{alg:3} recovers this value up to
the prescribed numerical tolerance.

This example also illustrates a case in which maximizing the proposed growth
bound is equivalent to maximizing the MAF. Every autocatalytic subhypergraph
of the reduced Oregonator must contain reaction $\mathrm R_3$. Moreover, its
effective kinetic constant, $\kappa_3^{\mathrm{eff}}=2.016\,\mathrm{s}^{-1}$,
is larger than the kinetic consumption associated with either
$\mathrm R_1$ or $\mathrm R_5$. Consequently, every feasible
autocatalytic candidate has the same kinetic norm,
$$
\|\mathbb S|_{\mathcal A'}\|_{\kappa}=2.016\,\mathrm{s}^{-1}.
$$
By Proposition~\ref{prop:reducciones}, maximizing the proposed growth bound
therefore reduces to maximizing the MAF. Algorithm~\ref{alg:3} consequently
selects the complete cycle
$\{\mathrm R_1,\mathrm R_3,\mathrm R_5\}$, whose amplification factor is
$\alpha^*=2.359304\ldots$.

Table~\ref{tab:oregonator_resultados} compares the theoretical growth bound
with the balanced-growth rate $\Lambda$, computed as the dominant eigenvalue
of the matrix $M$ in~\eqref{eq:operador-M}. For the isolated reaction
$\mathrm R_3$, the bound is attained exactly. For the complete cycle, the
additional pathway increases the stoichiometric amplification but produces
only a small increase in $\Lambda$, because regeneration through $Z$ and $Y$
is limited by the slower reactions $\mathrm R_5$ and $\mathrm R_1$.
Therefore, the example verifies the amplification factor returned by the
algorithm while also illustrating that a larger MAF does not necessarily
translate proportionally into a larger realized balanced-growth rate.
\begin{table}[h]
\centering
\begin{tabular}{lccccc}
\toprule
Subhypergraph &
$\alpha^*$ &
$\|\mathbb S|_{\mathcal A'}\|_{\kappa}$ &
Growth bound &
$\Lambda$ &
Bound status\\
\midrule
Reaction $\mathrm R_3$ &
$2$ &
$2.016$ &
$2.016$ &
$2.016$ &
Tight\\
Cycle $\{\mathrm R_1,\mathrm R_3,\mathrm R_5\}$ &
$2.3593$ &
$2.016$ &
$2.740$ &
$2.0196$ &
Not tight\\
\bottomrule
\end{tabular}
\caption{Oregonator results. The growth bound is
$(\alpha^*-1)\|\mathbb S|_{\mathcal A'}\|_{\kappa}$, while $\Lambda$ is
computed as the dominant eigenvalue of the matrix $M$ defined
in~\eqref{eq:operador-M}.}
\label{tab:oregonator_resultados}
\end{table}

\section{Computational analysis}\label{sec:exp}



This section investigates both the chemical relevance of the proposed
growth-bound criterion and the computational behavior of the corresponding
optimization framework. We first analyze three chemically relevant reaction
networks: the formose reaction, the reverse tricarboxylic acid cycle, and the
glyoxylate cycle. These examples illustrate situations in which the proposed
criterion either differs from or coincides with classical MAF maximization,
and clarify when the theoretical growth bound provides an informative
approximation of the realized balanced-growth rate. Finally, synthetic
multihypergraphs simulating CRNs of increasing size are used to assess the
computational scalability of the proposed algorithm.

The proposed parametric sweep (Algorithm~\ref{alg:3}) and the MAF
framework~\citep{gonzalez2024maf} were implemented in Python~3.12.8 using
\texttt{gurobipy}, with Gurobi~12.0.3 as the MILP solver for the auxiliary
problems~\eqref{eq:milp-aux}. To obtain the finite search space described in
Section~\ref{sec:sweep}, the kinetic constants were converted into scaled integer
coefficients by rounding them to a prescribed decimal precision and
multiplying by a common denominator. The parametric sweep then visited the
attainable scaled kinetic norms in decreasing order, each of them obtained by
solving the auxiliary problem~\eqref{eq:next-norm}. Whenever
the global MAF was certified, incumbent-based screening was activated to prune
non-improving norm levels. Otherwise, the complete scaled grid was explored,
which also preserves exactness.

Throughout the experiments, the flow upper bound was fixed to
$$
\Delta=
\sum_{a\in\mathcal A}\sum_{v\in\mathcal N}
(\mathbb T_{va}+\mathbb S_{va}),
$$
and the support threshold was set to
$\varepsilon=1/\Delta$ according to
constraint~\eqref{eq:P-supp}. Since the optimization is carried out on the
scaled instance, the reported kinetic norms and growth-bound values were
recomputed using the original kinetic constants on the selected
subhypergraph. Consequently, the optimization certificates apply to the scaled
instance, whereas all reported numerical values correspond to the original
kinetic data.

All experiments were carried out on the Euler Cluster (CeMEAI--ICMC/USP). The
corresponding hardware specifications and solver parameters are reported
below.

Throughout this section, the quality of the proposed selection criterion is
evaluated by comparing the theoretical growth bound with the realized
balanced-growth rate $\Lambda$, computed as the dominant eigenvalue of the
growth operator $M$. Optimization is performed on the scaled kinetic instance,
whereas all reported values are recomputed using the original kinetic
constants.

\subsection{The formose reaction network}
\label{sec:formose}

The formose reaction provides the largest chemically realistic reaction
network considered in this work and illustrates the limitations of the
proposed growth bound as a surrogate for realized balanced growth.

The benchmark is the formose reaction network from
\cite{muller2022formose}, a classical model for the autocatalytic production of
sugars from formaldehyde. After removing formaldehyde, which acts as an
external food source with fixed concentration, the resulting hypergraph
contains $28$ internal nodes and $38$ hyperarcs. The $17$ reactions consuming
formaldehyde become pseudo first order after incorporating the external
concentration into the corresponding effective kinetic constants, while the
remaining reactions are already first order. Consequently, the resulting
network naturally satisfies the scalable flow assumptions introduced in our
settings, and no reaction has to be discarded. Moreover, for each of the
$50$ kinetic realizations described below, we verified a posteriori that the
optimal subhypergraph satisfies $\mathcal{M}=\mathcal{N}'$, and hence
constraint~\eqref{eq:P-autonomy}, that it is connected, and that the dominant
eigenvector of its growth operator $M$ is strictly positive. Therefore, all
hypotheses of Theorem~\ref{thm:cotas} hold for every reported solution, and the
realized growth rate $\Lambda$ is well defined as a balanced-growth rate with
strictly positive composition.

The first observation concerns the stoichiometric structure of the network.
Applying the optimization framework of \cite{gonzalez2024maf} yields a maximum
amplification factor $\alpha^*=1.1382$. Moreover, the proposed parametric sweep
reveals that $\alpha^*$ remains essentially constant along the entire Pareto
frontier shown in Figure~\ref{fig:formosa_pareto}. Since the amplification
factor is purely determined by the incidence structure of the hypergraph, this
behavior is independent of the kinetic constants. Consequently, the objective
$(\alpha^*-1)\|\mathbb S|_{\mathcal A'}\|_{\kappa}$ is driven almost entirely by
the kinetic norm. As predicted by Proposition~\ref{prop:reducciones}, maximizing
the proposed growth bound therefore becomes practically equivalent to maximizing
the kinetic consumption rate.

\begin{figure}[h]
\centering
\includegraphics[width=0.6\textwidth,
trim={0.5cm 0.45cm 0.35cm 0.35cm},clip]
{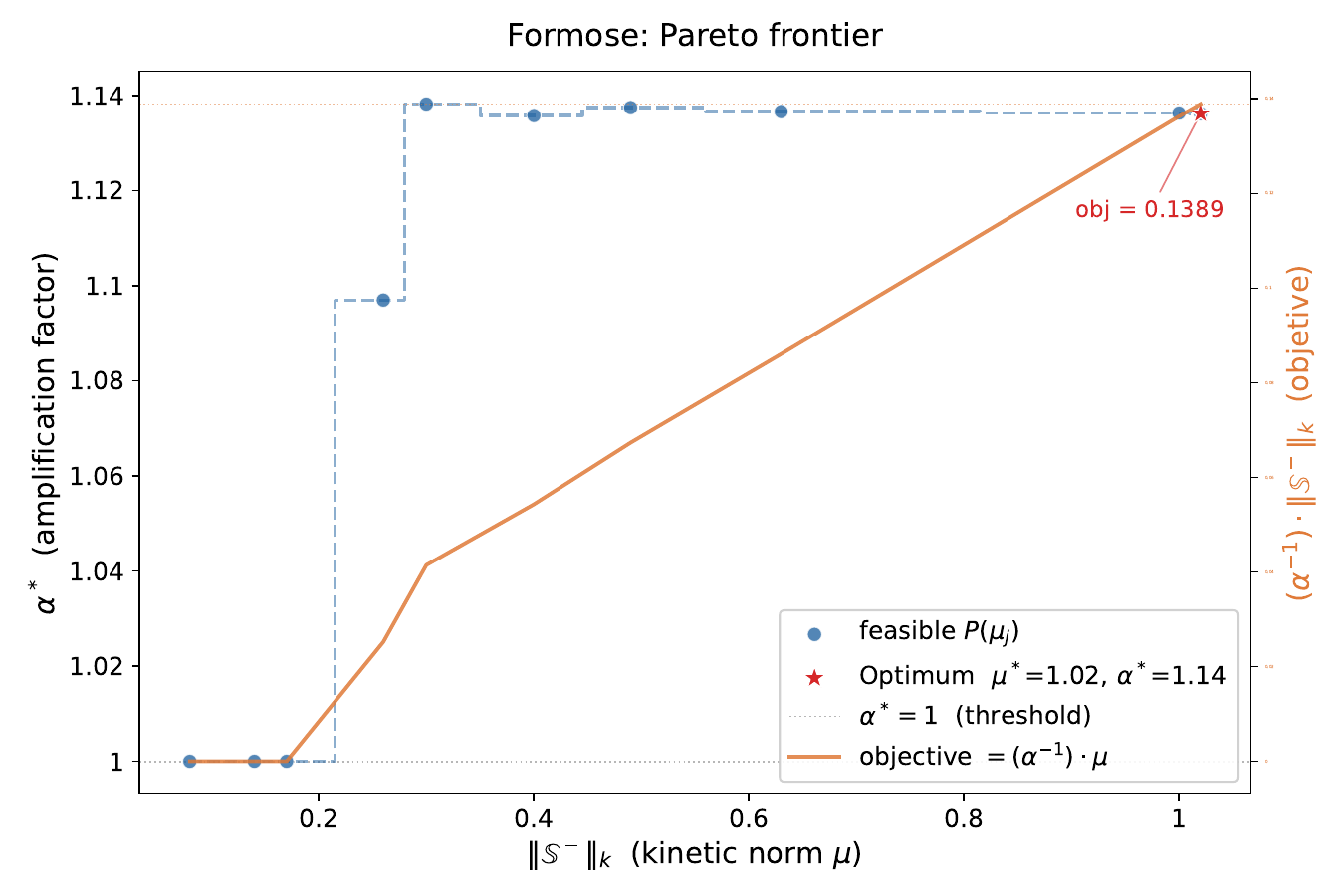}
\caption{Pareto frontier relating the Maximum Amplification Factor and the
kinetic norm for the formose reaction. Since the amplification factor rapidly
reaches its stoichiometric limit, the proposed objective is dominated by the
kinetic norm over most of the frontier.}
\label{fig:formosa_pareto}
\end{figure}

Unlike the Oregonator, the formose reaction is not accompanied by a complete set
of experimentally measured kinetic constants. Rather than selecting an
arbitrary parameter vector, we therefore adopt an ensemble approach. Reactions
are classified according to their stoichiometric type, namely aldol additions,
retro aldol reactions, isomerizations, and degradation reactions, and kinetic
constants are generated from activation energies reported in the computational
chemistry literature
\cite{kua2024formose,venturini2024formose}. For each reaction, the activation
energy is sampled uniformly within the interval associated with its reaction
class, converted into a kinetic constant using transition state theory, and
subsequently normalized. We generate $50$ independent kinetic realizations,
each corresponding to a plausible catalytic regime at fixed formaldehyde
concentration. For every realization we compute the bound optimal
autocatalytic subhypergraph, its realized balanced growth rate $\Lambda$, and
the corresponding theoretical growth bound.

The computational results reveal two important limitations of the proposed
criterion. First, although the theoretical bound of
Theorem~\ref{thm:cotas} is always satisfied, it is generally far from tight.
Figure~\ref{fig:formosa_holgura} reports the distribution of the ratio between
the realized growth rate and the theoretical upper bound over the entire
ensemble. The median value is approximately $0.003$, while the maximum observed
ratio is only $0.017$, indicating that the theoretical bound typically
overestimates the actual growth rate by more than two orders of magnitude.

This behavior originates from the separation between the reaction determining
the kinetic norm and the rate-limiting reaction governing the asymptotic growth
rate. In the formose network, the kinetic norm is determined by the fastest
aldol reactions, whereas balanced growth is controlled by the considerably
slower retro aldol reactions responsible for closing the autocatalytic cycle.
Consequently, the kinetic norm no longer provides a good surrogate for the
actual kinetic bottleneck.

\begin{figure}[h]
\centering
\includegraphics[width=0.62\textwidth,
trim={0.35cm 0.5cm 0.35cm 0.35cm},clip]
{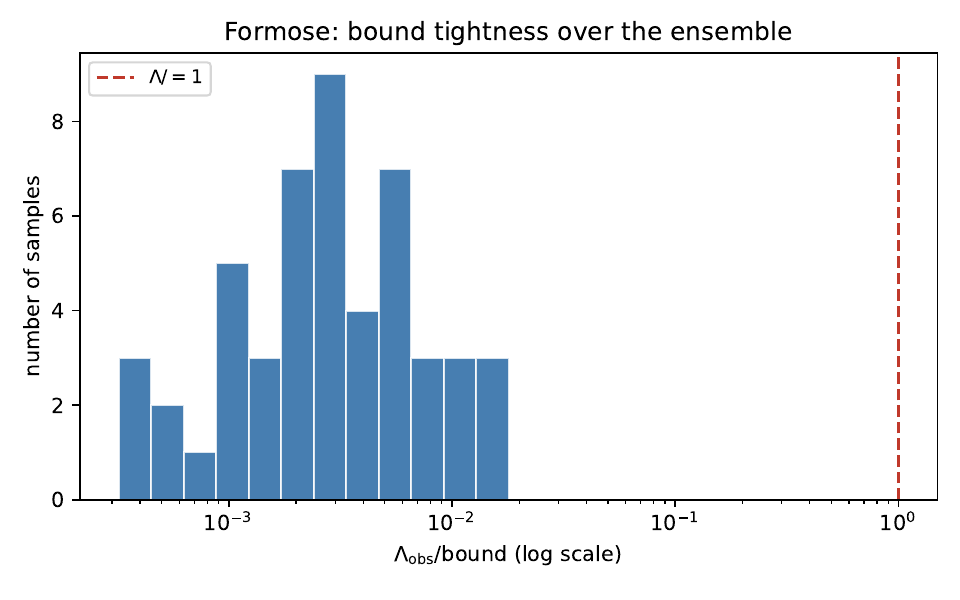}
\caption{The growth bound consistently overestimates the realized balanced-growth rate,
revealing that the fastest consuming reaction does not determine the kinetic
bottleneck of the autocatalytic cycle.}
\label{fig:formosa_holgura}
\end{figure}

The second limitation concerns the use of the growth bound as an optimization
criterion. Since the amplification factor remains almost constant throughout
the feasible region, the optimization model effectively ranks candidate
subhypergraphs according to the kinetic norm alone. As a consequence, the
selected subhypergraph does not necessarily maximize the realized growth rate.
Compared with the MAF optimal solution, the proposed criterion produces
slightly larger realized growth on average, with a median improvement of
approximately $2.4\%$. However, in $42\%$ of the sampled kinetic realizations it
selects a subhypergraph with smaller realized growth, and in the worst cases
the realized growth rate is more than an order of magnitude smaller than that of
the MAF optimal solution.

The selected subhypergraph is also sensitive to the kinetic realization.
Across the $50$ sampled parameter vectors, the optimization framework selects
nine distinct autocatalytic subhypergraphs, none of which appears in more than
$18\%$ of the instances, as illustrated in
Figure~\ref{fig:formosa_cores}. These results indicate that when the
stoichiometric amplification is essentially constant, the proposed growth bound
becomes largely controlled by the kinetic norm, whose maximizer is not
necessarily aligned with the actual kinetic bottleneck governing balanced
growth.

\begin{figure}[h]
\centering
\includegraphics[width=0.6\textwidth,
trim={0.35cm 0.45cm 0.35cm 0.35cm},clip]
{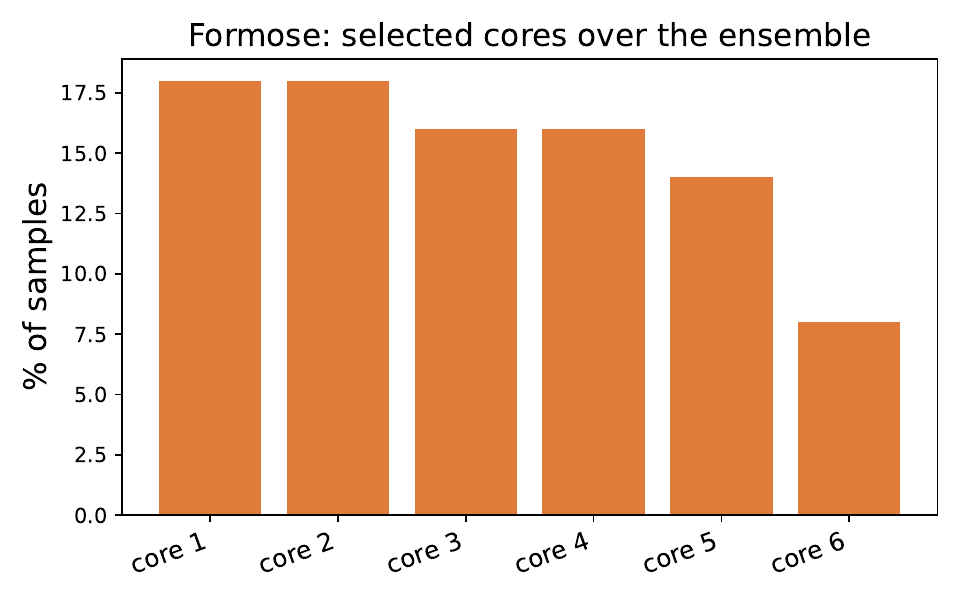}
\caption{Selection frequency of the five most frequently identified
autocatalytic subhypergraphs over the $50$ kinetic realizations of the
formose reaction. Nine different optimal subhypergraphs are selected, revealing
the sensitivity of the optimization criterion when the amplification factor is
essentially constant.}
\label{fig:formosa_cores}
\end{figure}
Finally, this network also illustrates the effect of the incumbent-based
screening. In all sampled kinetic realizations, the incumbent was found at one
of the largest scaled norm levels explored by the descending sweep. Whenever
the global MAF bound was certified, the screening test then pruned the
remaining lower norm levels immediately: across the $50$ realizations the
algorithm solved a median of two parametric subproblems, and never more than
eight, for a total of $134$ subproblems over the whole ensemble. 
This confirms that, even on
a network with $38$ hyperarcs, the outer search is dominated by very few
subproblems rather than by an exhaustive enumeration of the attainable norms.

\subsection{The rTCA and glyoxylate cycles}
\label{sec:carbono}

To determine whether the observations made for the formose reaction are
specific to that chemistry, we next consider two biologically relevant
autocatalytic pathways, 
namely the reverse tricarboxylic acid (rTCA) cycle, comprising eleven internal
species~\citep{smith2004universality}, and the glyoxylate cycle, comprising
seven internal species and known to be autocatalytic in the sense of
\cite{barenholz2017autocatalytic}. As in the previous experiments, external
metabolites such as $\mathrm{CO_2}$, reducing agents, ATP, and CoA are treated
as food species with fixed concentrations. As in the formose reaction,
no hyperarc needs to be removed. After fixing the external species, every
hyperarc is naturally described by a first-order scalable flow law involving a
single internal source node.

The computational results indicate that the lack of tightness observed in the
formose reaction is not specific to sugar chemistry. Under plausible catalytic
kinetics, generated using the same ensemble methodology described in the
previous subsection, the median ratios
$\Lambda/\bigl((\alpha^*-1)\|\mathbb S|_{\mathcal A'}\|_{\kappa}\bigr)$
are approximately $3.6\times10^{-3}$ for the rTCA cycle and
 $1.2\times10^{-2}$ for the glyoxylate cycle, one and four times the
median obtained for the formose network, respectively. In both pathways, the
kinetic norm is determined by a fast consumption process, whereas the realized
growth rate is controlled by a different rate-limiting reaction. Consequently,
the theoretical upper bound remains valid but substantially overestimates the
realized balanced growth rate. The glyoxylate cycle is particularly
illustrative because this behavior appears even in the absence of carbon
fixation, indicating that the discrepancy is not caused by carboxylation
reactions but rather by the separation between the fastest consuming species and
the kinetic bottleneck governing the autocatalytic cycle.

From the viewpoint of subhypergraph selection, however, these metabolic
networks behave very differently from the formose reaction. Although the
parametric sweep encounters a few attainable values of the kinetic norm, the
amplification factor exceeds one at only a single value: the complete cycle in
the rTCA network ($\alpha^*=1.0764$) and the short autocatalytic loop in the
glyoxylate cycle ($\alpha^*=1.1347$). At every other attainable norm the
maximum amplification factor equals one, so the growth-bound objective
vanishes. Consequently, the growth bound is maximized by a single
autocatalytic core, and maximizing the proposed bound coincides with
maximizing the MAF, in agreement with Proposition~\ref{prop:reducciones}(i), 
both optimization criteria return the same solution. In this sense these
pathways admit a unique bound-optimal core, over which the two criteria cannot
disagree.

These experiments also clarify the scope of the conclusions drawn from the
designed example and the formose reaction. A discrepancy between the proposed
growth bound and the MAF can only arise in networks possessing a sufficiently
rich combinatorial structure, namely multiple feasible autocatalytic
subhypergraphs with different kinetic norms. Such a situation occurs in the
formose reaction but not in compact metabolic cycles such as the rTCA and
glyoxylate pathways, where the optimization problem is essentially reduced to
the identification of a unique autocatalytic core.

Finally, the minimal autocatalytic core of the rTCA cycle, represented by a
three-node lumped cycle, provides an intermediate case between the illustrative
examples and the larger reaction networks. Its amplification factor equals the
plastic constant, that is, the real solution of
$\alpha^3=\alpha+1$, and the ratio
$\Lambda/\bigl((\alpha^*-1)\|\mathbb S|_{\mathcal A'}\|_{\kappa}\bigr)$
reaches approximately $0.46$ under nearly homogeneous kinetic constants. This
behavior is consistent with Lemma~\ref{lem:ajuste} and illustrates that the
growth bound becomes substantially tighter when the separation between the
fastest consuming reaction and the rate-limiting step largely disappears.


\subsection{Scalability analysis}
\label{sec:scalability}

The previous experiments assess the quality of the proposed growth bound as a
selection criterion. We now characterize the computational performance of
Algorithm~\ref{alg:3} on synthetic instances of increasing size.

Synthetic instances were generated by embedding a verified
autocatalytic core ($A\rightarrow2A$) into random CRNs of increasing size,
thereby guaranteeing feasibility while varying the surrounding network
structure. Instance sizes range from
$5\times7$ to $45\times60$, with five independent realizations per size. Kinetic constants
are sampled from a log-uniform distribution and subsequently normalized. Each
parametric subproblem $P(\mu)$ is solved with a time limit of $30$ seconds. For every instance we record the wall-clock time of the complete run, the number
of parametric subproblems solved by the descending scaled-norm sweep, the total
number of generalized fractional programming iterations, whether the global MAF
bound used for screening was certified, whether screening was activated, and
whether all auxiliary subproblems were solved to proven optimality.
All experiments were performed on the Euler Cluster
(CeMEAI--ICMC/USP), using one compute node with eight CPU cores, Python~3.12.8,
and Gurobi~12.0.3.

\begin{table}[h]
\centering
\begin{tabular}{lcccc}
\toprule
$|\mathcal N|\times|\mathcal A|$
& Median subproblems
& Median iterations
& Median time (s) [min--max]
& Proven optimal\\
\midrule
$5\times7$  & 4  & 50  & $0.9$ [$0.1$--$2.4$]    & 5/5 \\
$8\times11$ & 7  & 48  & $0.9$ [$0.1$--$19.5$]   & 5/5 \\
$12\times16$ & 10 & 571 & $18.1$ [$0.2$--$34.1$]  & 5/5 \\
$16\times21$ & 12 & 222 & $9.7$ [$0.4$--$24.8$]   & 5/5 \\
$20\times26$ & 7  & 54  & $3.6$ [$0.2$--$95.8$]   & 5/5 \\
$28\times38$ & 3  & 24  & $9.9$ [$1.4$--$49.7$]   & 5/5 \\
$36\times48$ & 5  & 173 & $27.6$ [$0.6$--$187.8$] & 5/5 \\
$45\times60$ & 2  & 1   & $2.3$ [$0.9$--$184.3$]  & 5/5 \\
\bottomrule
\end{tabular}
\caption{Computational performance of Algorithm~\ref{alg:3} on synthetic
autocatalytic CRNs. Reported values are medians over five
instances of each size. ``Median subproblems'' is the number of parametric
subproblems solved by the descending scaled-norm sweep. All of them are
feasible, since the sweep visits only attainable kinetic norms. The column ``Proven
optimal'' reports how many of the five instances had all auxiliary subproblems
solved to certified optimality within the $30$-second time limit, only those
runs provide a complete global optimality certificate for the scaled instance.
The $28\times38$ instance matches the size of the formose network.}
\label{tab:escalabilidad}
\end{table}

\begin{figure}[h]
\centering
\includegraphics[width=\textwidth]{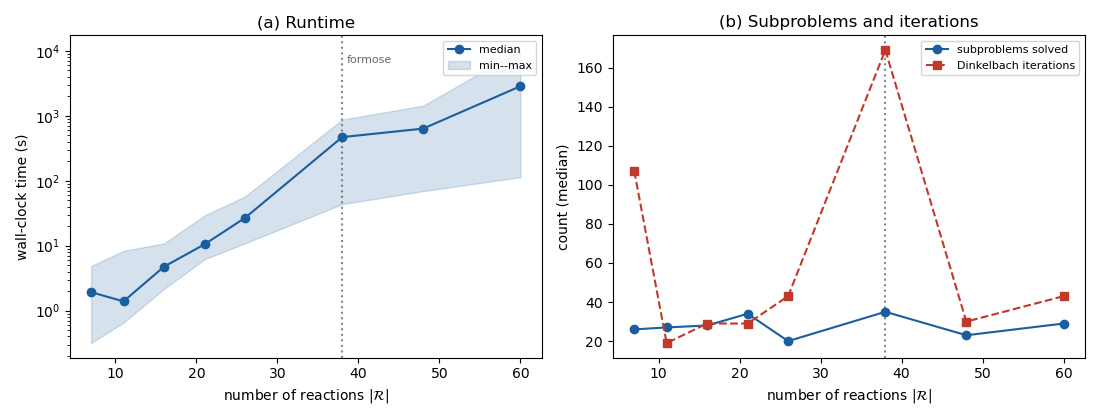}
\caption{Scalability of Algorithm~\ref{alg:3} as the number of hyperarcs
increases. Values are medians over the instances of each size, with the shaded
region showing the minimum and maximum. (a) Total wall-clock time. (b) Number of
parametric subproblems solved under screening and cumulative generalized
fractional programming iterations. The dotted vertical line marks the size of the
formose network.}
\label{fig:escalabilidad}
\end{figure}

 Figure~\ref{fig:escalabilidad} and Table~\ref{tab:escalabilidad} summarize the computational results. All instances are solved to certified optimality for the scaled formulation. Median solution times range from about one second for the smallest instances to about half a minute for the largest ones, although the variability between instances of the same size is substantial.

The descending sweep explores only attainable kinetic norms. As a result, the number of parametric subproblems remains small across all instance sizes, while most of the computational effort is concentrated in the auxiliary mixed-integer problems solved within the generalized fractional programming procedure. 

Overall, the proposed framework solves instances up to $45\times60$ to certified optimality within a few minutes, well beyond the size of the formose network.

\section{Conclusions}
\label{sec:conc}

This paper extends the kinetic--stoichiometric growth bound of
\citet{gagrani2026topological}, originally formulated for the analysis of
fixed reaction networks, into an optimization framework for selecting
reaction subnetworks with the greatest theoretical growth potential. The
optimal value of the resulting problem certifies an upper bound on the
balanced-growth rate of every autocatalytic subnetwork of a given network,
which provides, for instance, a necessary condition for the persistence of
autocatalytic chemistry in open reactors under dilution.

The resulting optimization problem is nonconvex. To solve it exactly, we
developed a parametric solution method that combines generalized fractional
programming with mixed-integer optimization. By exploiting the discrete
structure of the kinetic norm, the algorithm decomposes the original problem
into a finite sequence of MAF optimization problems, for which finite
convergence and global optimality are established for the corresponding scaled
formulation.

The computational analysis highlights both the strengths and the limitations of
the proposed criterion. On the Oregonator and the designed benchmark,
growth-bound maximization identifies the kinetically preferred
autocatalytic subnetwork, even when it differs from the MAF-optimal
solution. In contrast, the formose reaction together with the reverse
tricarboxylic acid and glyoxylate cycles show that the theoretical bound may
substantially overestimate the realized balanced-growth rate whenever the
reaction determining the kinetic norm is decoupled from the rate-limiting step
governing growth. Consequently, the proposed criterion is most informative for
reaction networks containing multiple competing autocatalytic subnetworks with
distinct kinetic characteristics, whereas structurally rigid autocatalytic
cycles naturally reduce the optimization problem to classical MAF
maximization. To our knowledge, these constitute the first computational
assessment of consumption-based growth bounds on chemically relevant reaction
networks, thereby clarifying both their applicability and their limitations.

From an algorithmic perspective, the scalability study shows that all synthetic instances, up to $45\times60$, are solved to certified optimality within a few minutes. The descending sweep visits only a small number of attainable kinetic norms, while incumbent-based screening further reduces the number of parametric subproblems that must be solved.

For process systems engineering, the framework provides a rigorous basis for comparing the growth bounds of candidate autocatalytic subnetworks. Its use in reaction-network screening should be complemented by detailed kinetic analysis when assessing attainable performance, since maximizing an upper bound does not guarantee maximal realized growth.

Several directions naturally emerge from this work. From a modeling
perspective, developing kinetic norms that explicitly account for the
rate-limiting step could substantially tighten the theoretical bound and
improve its predictive ability. Since the bound of Theorem~\ref{thm:cotas} holds for mass-action
kinetics of arbitrary order in the linear-volume formulation, the optimization
model already applies beyond first-order kinetics. Validating its predictions
in that setting requires computing positive fixed points of the nonlinear
composition dynamics, which would considerably broaden the applicability of
the analysis to biochemical and metabolic reaction networks.
From an optimization perspective, decomposition methods, stronger
formulations, valid inequalities, and specialized heuristics for the
auxiliary mixed-integer problems appear to be the most promising directions
for improving scalability.

Overall, this work establishes an exact optimization framework for studying
kinetic--stoichiometric growth in chemical reaction networks, providing a
bridge between structural amplification theory, autocatalysis, reaction kinetics, and
mathematical optimization through the unified language of directed
multihypergraphs.

\section*{Code and Data Availability}
The code implementing Algorithm~\ref{alg:3} and the MAF framework
of~\citet{gonzalez2024maf}, together with the formose reaction stoichiometric
matrices~\citep{muller2022formose} and the generator for the synthetic
scalability instances, is publicly available at
\url{https://github.com/anticiclon/kinetic-stoichiometric-growth-bounds}.

\section*{Acknowledgements}
This research was supported by grants PID2020-114594GB-C21, PID2024-156594NB-C21, IMAG-Mar\'ia de Maeztu CEX2020-001105-M funded by MCIN/AEI/10.13039/501100011033, the project C‐EXP‐139‐UGR23 funded by Programa de Andaluc\'ia FEDER, and grant~2026/00417-9
(FAPESP). The computational experiments were carried 
out using the computational resources of the Center for 
Mathematical Sciences Applied to Industry (CeMEAI) funded by 
FAPESP (grant~2013/07375-0).

%

\begin{thebibliography}{41}
\expandafter\ifx\csname natexlab\endcsname\relax\def\natexlab#1{#1}\fi
\providecommand{\url}[1]{\texttt{#1}}
\providecommand{\href}[2]{#2}
\providecommand{\path}[1]{#1}
\providecommand{\DOIprefix}{doi:\space}
\providecommand{\ArXivprefix}{arXiv:}
\providecommand{\URLprefix}{URL: }
\providecommand{\Pubmedprefix}{pmid:}
\providecommand{\doi}[1]{\href{https://doi.org/#1}{\path{#1}}}
\providecommand{\Pubmed}[1]{\href{pmid:#1}{\path{#1}}}
\providecommand{\bibinfo}[2]{#2}
\ifx\xfnm\relax \def\xfnm[#1]{\unskip,\space#1}\fi
\bibitem[{Acemoglu et~al.(2012)Acemoglu, Carvalho, Ozdaglar and
  Tahbaz-Salehi}]{acemoglu2012network}
\bibinfo{author}{Acemoglu, D.}, \bibinfo{author}{Carvalho, V.M.},
  \bibinfo{author}{Ozdaglar, A.}, \bibinfo{author}{Tahbaz-Salehi, A.},
  \bibinfo{year}{2012}.
\newblock \bibinfo{title}{The network origins of aggregate fluctuations}.
\newblock \bibinfo{journal}{Econometrica} \bibinfo{volume}{80},
  \bibinfo{pages}{1977--2016}.
\bibitem[{Acemoglu et~al.(2017)Acemoglu, Ozdaglar and
  Tahbaz-Salehi}]{acemoglu2017networks}
\bibinfo{author}{Acemoglu, D.}, \bibinfo{author}{Ozdaglar, A.},
  \bibinfo{author}{Tahbaz-Salehi, A.}, \bibinfo{year}{2017}.
\newblock \bibinfo{title}{Microeconomic origins of macroeconomic tail risks}.
\newblock \bibinfo{journal}{American Economic Review} \bibinfo{volume}{107},
  \bibinfo{pages}{54--108}.
\bibitem[{Aghion and Howitt(2009)}]{aghion2008economics}
\bibinfo{author}{Aghion, P.}, \bibinfo{author}{Howitt, P.W.},
  \bibinfo{year}{2009}.
\newblock \bibinfo{title}{{The Economics of Growth}}.
\newblock \bibinfo{publisher}{The MIT Press}, \bibinfo{address}{Cambridge, MA
  and London}.
\newblock \bibinfo{note}{With the collaboration of Leonardo Bursztyn}.
\bibitem[{Ahuja et~al.(1988)Ahuja, Magnanti and Orlin}]{ahuja1988network}
\bibinfo{author}{Ahuja, R.K.}, \bibinfo{author}{Magnanti, T.L.},
  \bibinfo{author}{Orlin, J.B.}, \bibinfo{year}{1988}.
\newblock \bibinfo{title}{Network flows}.
\newblock Working Paper 2059-88, Alfred P. Sloan School of Management,
  Massachusetts Institute of Technology, Cambridge, MA.
\newblock \URLprefix\url{https://hdl.handle.net/1721.1/49424}.
\bibitem[{Atamt{\"u}rk and G{\'o}mez(2020)}]{atamturk2020strong}
\bibinfo{author}{Atamt{\"u}rk, A.}, \bibinfo{author}{G{\'o}mez, A.},
  \bibinfo{year}{2020}.
\newblock \bibinfo{title}{Strong formulations for sparse optimization}.
\newblock \bibinfo{journal}{Mathematical Programming} \bibinfo{volume}{181},
  \bibinfo{pages}{63--95}.
\newblock \DOIprefix\doi{10.1007/s10107-019-01407-9}.
\bibitem[{Baqaee and Farhi(2019)}]{baqaee2019granular}
\bibinfo{author}{Baqaee, D.R.}, \bibinfo{author}{Farhi, E.},
  \bibinfo{year}{2019}.
\newblock \bibinfo{title}{The macroeconomic impact of microeconomic shocks:
  Beyond Hulten's theorem}.
\newblock \bibinfo{journal}{Econometrica} \bibinfo{volume}{87},
  \bibinfo{pages}{1155--1203}.
\bibitem[{Barenholz et~al.(2017)Barenholz, Davidi, Reznik, Bar-On, Antonovsky,
  Noor and Milo}]{barenholz2017autocatalytic}
\bibinfo{author}{Barenholz, U.}, \bibinfo{author}{Davidi, D.},
  \bibinfo{author}{Reznik, E.}, \bibinfo{author}{Bar-On, Y.},
  \bibinfo{author}{Antonovsky, N.}, \bibinfo{author}{Noor, E.},
  \bibinfo{author}{Milo, R.}, \bibinfo{year}{2017}.
\newblock \bibinfo{title}{Design principles of autocatalytic cycles constrain
  enzyme kinetics and force low substrate saturation at flux branch points}.
\newblock \bibinfo{journal}{eLife} \bibinfo{volume}{6},
  \bibinfo{pages}{e20667}.
\bibitem[{Battiston et~al.(2021)Battiston, Amico, Barrat, Bianconi, Ferraz~de
  Arruda, Franceschiello, Iacopini, K{\'e}fi, Latora, Moreno
  et~al.}]{battiston2021physics}
\bibinfo{author}{Battiston, F.}, \bibinfo{author}{Amico, E.},
  \bibinfo{author}{Barrat, A.}, \bibinfo{author}{Bianconi, G.},
  \bibinfo{author}{Ferraz~de Arruda, G.}, \bibinfo{author}{Franceschiello, B.},
  \bibinfo{author}{Iacopini, I.}, \bibinfo{author}{K{\'e}fi, S.},
  \bibinfo{author}{Latora, V.}, \bibinfo{author}{Moreno, Y.}, et~al.,
  \bibinfo{year}{2021}.
\newblock \bibinfo{title}{The physics of higher-order interactions in complex
  systems}.
\newblock \bibinfo{journal}{Nature Physics} \bibinfo{volume}{17},
  \bibinfo{pages}{1093--1098}.
\bibitem[{Bertsimas et~al.(2016)Bertsimas, King and
  Mazumder}]{bertsimas2016best}
\bibinfo{author}{Bertsimas, D.}, \bibinfo{author}{King, A.},
  \bibinfo{author}{Mazumder, R.}, \bibinfo{year}{2016}.
\newblock \bibinfo{title}{Best subset selection via a modern optimization
  lens}.
\newblock \bibinfo{journal}{The Annals of Statistics} \bibinfo{volume}{44},
  \bibinfo{pages}{813--852}.
\newblock \DOIprefix\doi{10.1214/15-AOS1388}.
\bibitem[{{Bick} et~al.(2023){Bick}, Gross, Harrington and
  Schaub}]{bick2023higher}
\bibinfo{author}{{Bick}, C.}, \bibinfo{author}{Gross, E.},
  \bibinfo{author}{Harrington, H.A.}, \bibinfo{author}{Schaub, M.T.},
  \bibinfo{year}{2023}.
\newblock \bibinfo{title}{What are higher-order networks?}
\newblock \bibinfo{journal}{SIAM Review} \bibinfo{volume}{65},
  \bibinfo{pages}{686--731}.
\bibitem[{{Blanco} et~al.(2026){Blanco}, G{\'a}zquez and
  {Oca{\~n}a-Rivas}}]{blanco2026constructing}
\bibinfo{author}{{Blanco}, V.}, \bibinfo{author}{G{\'a}zquez, R.},
  \bibinfo{author}{{Oca{\~n}a-Rivas}, J.}, \bibinfo{year}{2026}.
\newblock \bibinfo{title}{Constructing nested self-amplifying multiperiod
  hypergraphs through mathematical optimization}.
\newblock \bibinfo{journal}{arXiv preprint arXiv:2604.12511}.
\bibitem[{{Blanco} et~al.(2024){Blanco}, Gonz{\'a}lez and
  Gagrani}]{gonzalez2024maf}
\bibinfo{author}{{Blanco}, V.}, \bibinfo{author}{Gonz{\'a}lez, G.},
  \bibinfo{author}{Gagrani, P.}, \bibinfo{year}{2024}.
\newblock \bibinfo{title}{Identifying self-amplifying hypergraph structures
  through mathematical optimization}.
\newblock \bibinfo{journal}{arXiv preprint arXiv:2412.15776}
  \bibinfo{note}{Submitted to Omega}.
\bibitem[{Blanco et~al.(2025)Blanco, Gonz{\'a}lez and Puerto}]{blanco2025fixed}
\bibinfo{author}{Blanco, V.}, \bibinfo{author}{Gonz{\'a}lez, G.},
  \bibinfo{author}{Puerto, J.}, \bibinfo{year}{2025}.
\newblock \bibinfo{title}{Fixed topology minimum-length trees with
  neighborhoods: A Steiner tree based approach}.
\newblock \bibinfo{journal}{Computers \& Industrial Engineering}
  \bibinfo{volume}{207}, \bibinfo{pages}{111331}.
\bibitem[{Blokhuis et~al.(2020)Blokhuis, Lacoste and
  Nghe}]{blokhuis2020universal}
\bibinfo{author}{Blokhuis, A.}, \bibinfo{author}{Lacoste, D.},
  \bibinfo{author}{Nghe, P.}, \bibinfo{year}{2020}.
\newblock \bibinfo{title}{Universal motifs and the diversity of autocatalytic
  systems}.
\newblock \bibinfo{journal}{Proceedings of the National Academy of Sciences}
  \bibinfo{volume}{117}, \bibinfo{pages}{25230--25236}.
\bibitem[{Burgard et~al.(2003)Burgard, Pharkya and
  Maranas}]{burgard2003optknock}
\bibinfo{author}{Burgard, A.P.}, \bibinfo{author}{Pharkya, P.},
  \bibinfo{author}{Maranas, C.D.}, \bibinfo{year}{2003}.
\newblock \bibinfo{title}{OptKnock: A bilevel programming framework for
  identifying gene knockout strategies for microbial strain optimization}.
\newblock \bibinfo{journal}{Biotechnology and Bioengineering}
  \bibinfo{volume}{84}, \bibinfo{pages}{647--657}.
\newblock \DOIprefix\doi{10.1002/bit.10803}.
\bibitem[{Crouzeix et~al.(1985)Crouzeix, Ferland and Schaible}]{crouzeix1985}
\bibinfo{author}{Crouzeix, J.P.}, \bibinfo{author}{Ferland, J.A.},
  \bibinfo{author}{Schaible, S.}, \bibinfo{year}{1985}.
\newblock \bibinfo{title}{An algorithm for generalized fractional programs}.
\newblock \bibinfo{journal}{Journal of Optimization Theory and Applications}
  \bibinfo{volume}{47}, \bibinfo{pages}{35--49}.
\bibitem[{Dinkelbach(1967)}]{dinkelbach1967}
\bibinfo{author}{Dinkelbach, W.}, \bibinfo{year}{1967}.
\newblock \bibinfo{title}{On nonlinear fractional programming}.
\newblock \bibinfo{journal}{Management Science} \bibinfo{volume}{13},
  \bibinfo{pages}{492--498}.
\newblock \URLprefix \url{http://www.jstor.org/stable/2627691}.
\bibitem[{Feinberg(2019)}]{feinberg2019foundations}
\bibinfo{author}{Feinberg, M.}, \bibinfo{year}{2019}.
\newblock \bibinfo{title}{Foundations of Chemical Reaction Network Theory}.
  Volume \bibinfo{volume}{202} of \textit{\bibinfo{series}{Applied Mathematical
  Sciences}}.
\newblock \bibinfo{publisher}{Springer},
  \bibinfo{address}{Cham}.
\newblock \DOIprefix\doi{10.1007/978-3-030-03858-8}.
\bibitem[{Field and Noyes(1974)}]{field1974oscillations}
\bibinfo{author}{Field, R.J.}, \bibinfo{author}{Noyes, R.M.},
  \bibinfo{year}{1974}.
\newblock \bibinfo{title}{Oscillations in chemical systems. IV. Limit cycle
  behavior in a model of a real chemical reaction}.
\newblock \bibinfo{journal}{The Journal of Chemical Physics}
  \bibinfo{volume}{60}, \bibinfo{pages}{1877--1884}.
\bibitem[{Fischetti and Jo(2018)}]{fischetti2018deep}
\bibinfo{author}{Fischetti, M.}, \bibinfo{author}{Jo, J.},
  \bibinfo{year}{2018}.
\newblock \bibinfo{title}{Deep neural networks as 0--1 mixed integer linear
  programs: A feasibility study}.
\newblock \bibinfo{journal}{Constraints} \bibinfo{volume}{23},
  \bibinfo{pages}{296--309}.
\newblock \DOIprefix\doi{10.1007/s10601-018-9285-6}.
\bibitem[{Gagrani et~al.(2024)Gagrani, Blanco, Smith and Baum}]{Gagrani2024}
\bibinfo{author}{Gagrani, P.}, \bibinfo{author}{Blanco, V.},
  \bibinfo{author}{Smith, E.}, \bibinfo{author}{Baum, D.},
  \bibinfo{year}{2024}.
\newblock \bibinfo{title}{Polyhedral geometry and combinatorics of an
  autocatalytic ecosystem}.
\newblock \bibinfo{journal}{Journal of Mathematical Chemistry}
  \bibinfo{volume}{62}, \bibinfo{pages}{1012--1078}.
\bibitem[{Gagrani et~al.(2026)Gagrani, Wang, De~Decker and
  Lacoste}]{gagrani2026topological}
\bibinfo{author}{Gagrani, P.}, \bibinfo{author}{Wang, J.},
  \bibinfo{author}{De~Decker, Y.}, \bibinfo{author}{Lacoste, D.},
  \bibinfo{year}{2026}.
\newblock \bibinfo{title}{Topological bounds on the dynamical growth rate of
  chemical reaction networks}.
\newblock \bibinfo{journal}{arXiv preprint arXiv:2603.02627}.
\bibitem[{Hordijk and Steel(2004)}]{hordijk2004detecting}
\bibinfo{author}{Hordijk, W.}, \bibinfo{author}{Steel, M.},
  \bibinfo{year}{2004}.
\newblock \bibinfo{title}{Detecting autocatalytic, self-sustaining sets in
  chemical reaction systems}.
\newblock \bibinfo{journal}{Journal of Theoretical Biology}
  \bibinfo{volume}{227}, \bibinfo{pages}{451--461}.
\bibitem[{Hordijk and Steel(2018)}]{hordijk2018autocatalytic}
\bibinfo{author}{Hordijk, W.}, \bibinfo{author}{Steel, M.},
  \bibinfo{year}{2018}.
\newblock \bibinfo{title}{Autocatalytic networks at the basis of life’s
  origin and organization}.
\newblock \bibinfo{journal}{Life} \bibinfo{volume}{8}, \bibinfo{pages}{62}.
\bibitem[{Huang and Ulanowicz(2014)}]{huang2014ecological}
\bibinfo{author}{Huang, J.}, \bibinfo{author}{Ulanowicz, R.E.},
  \bibinfo{year}{2014}.
\newblock \bibinfo{title}{Ecological network analysis for economic systems:
  growth and development and implications for sustainable development}.
\newblock \bibinfo{journal}{PLoS ONE} \bibinfo{volume}{9},
  \bibinfo{pages}{e100923}.
\bibitem[{Kauffman(1986)}]{kauffman1986}
\bibinfo{author}{Kauffman, S.A.}, \bibinfo{year}{1986}.
\newblock \bibinfo{title}{Autocatalytic sets of proteins}.
\newblock \bibinfo{journal}{Journal of Theoretical Biology}
  \bibinfo{volume}{119}, \bibinfo{pages}{1--24}.
\bibitem[{Klamt et~al.(2009)Klamt, Haus and Theis}]{klamt2009hypergraphs}
\bibinfo{author}{Klamt, S.}, \bibinfo{author}{Haus, U.U.},
  \bibinfo{author}{Theis, F.}, \bibinfo{year}{2009}.
\newblock \bibinfo{title}{Hypergraphs and cellular networks}.
\newblock \bibinfo{journal}{PLoS Computational Biology} \bibinfo{volume}{5},
  \bibinfo{pages}{e1000385}.
\bibitem[{Kua and Tripoli(2024)}]{kua2024formose}
\bibinfo{author}{Kua, J.}, \bibinfo{author}{Tripoli, L.P.},
  \bibinfo{year}{2024}.
\newblock \bibinfo{title}{Exploring the core formose cycle: Catalysis and
  competition}.
\newblock \bibinfo{journal}{Life} \bibinfo{volume}{14}, \bibinfo{pages}{933}.
\newblock \DOIprefix\doi{10.3390/life14080933}.
\bibitem[{Labb{\'e} et~al.(2019)Labb{\'e}, Mart{\'\i}nez-Merino and
  Rodr{\'\i}guez-Ch{\'\i}a}]{labbe2019mixed}
\bibinfo{author}{Labb{\'e}, M.}, \bibinfo{author}{Mart{\'\i}nez-Merino, L.I.},
  \bibinfo{author}{Rodr{\'\i}guez-Ch{\'\i}a, A.M.}, \bibinfo{year}{2019}.
\newblock \bibinfo{title}{Mixed integer linear programming for feature
  selection in support vector machine}.
\newblock \bibinfo{journal}{Discrete Applied Mathematics}
  \bibinfo{volume}{261}, \bibinfo{pages}{276--304}.
\bibitem[{Laporte et~al.(2019)Laporte, Nickel and Saldanha~da Gama}]{LS2019}
\bibinfo{editor}{Laporte, G.}, \bibinfo{editor}{Nickel, S.},
  \bibinfo{editor}{Saldanha~da Gama, F.} (Eds.), \bibinfo{year}{2019}.
\newblock \bibinfo{title}{Location Science}.
\newblock \bibinfo{edition}{2nd} ed., \bibinfo{publisher}{Springer},
  \bibinfo{address}{Cham}.
\newblock \DOIprefix\doi{10.1007/978-3-030-32177-2}.
\bibitem[{Lente(2019)}]{LENTE2019587}
\bibinfo{author}{Lente, G.}, \bibinfo{year}{2019}.
\newblock A novel method to compute the time dependence of state distributions in the stochastic kinetic description of an autocatalytic system.
\newblock \bibinfo{journal}{Computers \& Chemical Engineering} \bibinfo{volume}{125}, \bibinfo{pages}{587--593}.
\newblock \DOIprefix\doi{10.1016/j.compchemeng.2016.08.001}.
\bibitem[{Magnanti and Wong(1984)}]{magnanti1984network}
\bibinfo{author}{Magnanti, T.L.}, \bibinfo{author}{Wong, R.T.},
  \bibinfo{year}{1984}.
\newblock \bibinfo{title}{Network design and transportation planning: Models
  and algorithms}.
\newblock \bibinfo{journal}{Transportation Science} \bibinfo{volume}{18},
  \bibinfo{pages}{1--55}.
\bibitem[{Müller et~al.(2022)Müller, Flamm and Stadler}]{muller2022formose}
\bibinfo{author}{Müller, S.}, \bibinfo{author}{Flamm, C.},
  \bibinfo{author}{Stadler, P.F.}, \bibinfo{year}{2022}.
\newblock \bibinfo{title}{What makes a reaction network ``chemical''?}
\newblock \bibinfo{journal}{Journal of Cheminformatics} \bibinfo{volume}{14},
  \bibinfo{pages}{63}.
\bibitem[{Nagurney(2006)}]{nagurney2013supply}
\bibinfo{author}{Nagurney, A.}, \bibinfo{year}{2006}.
\newblock \bibinfo{title}{Supply Chain Network Economics: Dynamics of Prices,
  Flows and Profits}.
\newblock \bibinfo{publisher}{Edward Elgar Publishing},
  \bibinfo{address}{Cheltenham, UK and Northampton, MA, USA}.
\bibitem[{von Neumann(1945)}]{neumann1945model}
\bibinfo{author}{von Neumann, J.}, \bibinfo{year}{1945}.
\newblock \bibinfo{title}{A model of general economic equilibrium}.
\newblock \bibinfo{journal}{The Review of Economic Studies}
  \bibinfo{volume}{13}, \bibinfo{pages}{1--9}.
\bibitem[{Orth et~al.(2010)Orth, Thiele and Palsson}]{orth2010reconstruction}
\bibinfo{author}{Orth, J.D.}, \bibinfo{author}{Thiele, I.},
  \bibinfo{author}{Palsson, B.{\O}.}, \bibinfo{year}{2010}.
\newblock \bibinfo{title}{What is flux balance analysis?}
\newblock \bibinfo{journal}{Nature Biotechnology} \bibinfo{volume}{28},
  \bibinfo{pages}{245--248}.
\newblock \DOIprefix\doi{10.1038/nbt.1614}.
\bibitem[{Smith and Morowitz(2004)}]{smith2004universality}
\bibinfo{author}{Smith, E.}, \bibinfo{author}{Morowitz, H.J.},
  \bibinfo{year}{2004}.
\newblock \bibinfo{title}{Universality in intermediary metabolism}.
\newblock \bibinfo{journal}{Proceedings of the National Academy of Sciences}
  \bibinfo{volume}{101}, \bibinfo{pages}{13168--13173}.
\bibitem[{Venturini and González(2024)}]{venturini2024formose}
\bibinfo{author}{Venturini, A.}, \bibinfo{author}{González, J.},
  \bibinfo{year}{2024}.
\newblock \bibinfo{title}{Prebiotic synthesis of glycolaldehyde and
  glyceraldehyde from formaldehyde: A computational study on the initial steps
  of the formose reaction}.
\newblock \bibinfo{journal}{ChemPlusChem} \bibinfo{volume}{89},
  \bibinfo{pages}{e202300388}.

\bibitem[{Willis and von Stosch(2016)}]{WILLIS201631}
\bibinfo{author}{Willis, M.J., von Stosch, M.}, \bibinfo{year}{2016}.
\newblock Inference of chemical reaction networks using mixed integer linear programming.
\newblock \bibinfo{journal}{Computers \& Chemical Engineering} \bibinfo{volume}{90}, \bibinfo{pages}{31--43}.
\newblock \DOIprefix\doi{10.1016/j.compchemeng.2016.04.019}.

\bibitem[{Wilson and Sahinidis(2019)}]{WILSON201988}
\bibinfo{author}{Wilson, Z.T., Sahinidis, N.V.}, \bibinfo{year}{2019}.
\newblock Automated learning of chemical reaction networks.
\newblock \bibinfo{journal}{Computers \& Chemical Engineering} \bibinfo{volume}{127}, \bibinfo{pages}{88--98}.
\newblock \DOIprefix\doi{10.1016/j.compchemeng.2019.05.020}.

\bibitem[{Woodley(2017)}]{WOODLEY2017297}
\bibinfo{author}{Woodley, J.M.}, \bibinfo{year}{2017}.
\newblock Bioprocess intensification for the effective production of chemical products.
\newblock \bibinfo{journal}{Computers \& Chemical Engineering} \bibinfo{volume}{105}, \bibinfo{pages}{297--307}.
\newblock \DOIprefix\doi{10.1016/j.compchemeng.2017.01.015}.

\end{thebibliography}

\end{document}